\documentclass[11pt,english,reqno]{amsart}
\usepackage[dvipsnames]{xcolor}
\usepackage{comment}
\usepackage{bm}
\usepackage{color}
\usepackage{babel}
\usepackage{verbatim}
\usepackage{amsthm}
\usepackage{amstext}
\usepackage{amssymb}
\usepackage{amsbsy}
\usepackage{ifpdf}
\usepackage[nobysame,abbrev,alphabetic]{amsrefs}
\usepackage{enumerate}
\usepackage{amsmath} 
\usepackage{amscd}
\usepackage{amsfonts}
\usepackage{graphicx}
\usepackage{verbatim}
\usepackage{gensymb}
\usepackage{mathrsfs}
\usepackage[right,displaymath,mathlines]{lineno}

\usepackage[original]{imakeidx}
\usepackage{tikz}
\usepackage{ifthen}
\usepackage{tikz-3dplot}
\usetikzlibrary{shapes.geometric}
\usetikzlibrary{decorations.pathreplacing}

\makeindex 

\usepackage[colorlinks]{hyperref}
\hypersetup{
linkcolor=blue,          
citecolor=green,        
} 

\usepackage[normalem]{ulem}

\usepackage[super]{nth}

\newtheorem{theorem}{Theorem}[section]
\newtheorem{claim}[theorem]{Claim}
\newtheorem{lemma}[theorem]{Lemma}

\newtheorem{corollary}[theorem]{Corollary}
\theoremstyle{definition}\newtheorem{definition}[theorem]{Definition}

\newtheorem{example}[theorem]{Example}

\newtheorem{proposition}[theorem]{Proposition}

\theoremstyle{definition}

\theoremstyle{definition}
\theoremstyle{definition}\newtheorem{remark}[theorem]{Remark}
\theoremstyle{definition}\newtheorem{convention}[theorem]{Convention}
\theoremstyle{definition}
\newcommand{\al}{\alpha}

\newcommand{\ka}{\kappa}
\newcommand{\sig}{\sigma}
\newcommand{\Sig}{\Sigma}

\newcommand{\bZ}{\mathbb{Z}}

\newcommand{\bF}{\mathbb{F}}

\newcommand{\bN}{\mathbb{N}}

\newcommand{\PGL}{\operatorname{PGL}}

\newcommand\wh[1]{\widehat{#1}}
\newcommand\set[1]{\left\{#1\right\}}
\newcommand\pa[1]{\left(#1\right)}
\newcommand\idist[1]{\left\langle#1\right\rangle}

\newcommand\av[1]{|#1|}
\newcommand\on[1]{\operatorname{#1}}

\newcommand\mb[1]{\mathbf{#1}}

\newcommand\mat[1]{\pa{\begin{matrix}#1\end{matrix}}}
\newcommand\br[1]{\left[#1\right]}
\newcommand\smallmat[1]{\pa{\begin{smallmatrix}#1\end{smallmatrix}}}

\newcommand{\Mat}{\operatorname{Mat}}

\newcommand{\lra}{\longrightarrow}

\newcommand{\onto}{\xymatrix{\ar@{>>}[r]&}}

\newcommand{\usnote}[1]{\marginpar{\color{cyan}\tiny [US] #1}}

\newcommand{\red}[1]{\textcolor{red}{#1}}

\newcommand {\ignore}[1]  {}

\newcommand{\tilo}{\mb{o}}
\newcommand{\tilec}[1]{\mb{c}_{#1}}

\newcommand{\tilea}{\mb{a}}
\newcommand{\tileb}{\mb{b}}
\newcommand{\tiled}{\mb{d}}
\newcommand{\tm}{\alpha}

\newcommand{\splus}[1]{\tilea_{#1}}
\newcommand{\sminus}[1]{\tileb_{#1}}
\newcommand{\tiler}[1]{\tiled_{\mb{#1}}}

\usepackage{amssymb,amsmath,amsthm}
\usepackage[utf8]{inputenc}
\usepackage[T1]{fontenc}
\usepackage{enumerate}
\usepackage{array}
\usepackage{longtable}
\usepackage{stmaryrd}
\usepackage{mathrsfs} 
\usepackage{xcolor}
\usepackage{mathtools}
\usepackage{mathdots}
\usepackage{nicefrac}
\usepackage{xfrac}

\def\F{\mathbb{F}}

\def\deg{\mathrm{deg}}

\def\sep{\,:\,}
\def\sub{\subseteq}
\def\defeq{:=}
\def\eqdef{=:}

\def\es{e}
\def\esc{E}
\def\height{h}

\def\bbf{\F}
\def\fq{\F_q}
\def\fqt{\fq[t]}
\def\fqtt{\fq(t)}
\def\fqttt{\fq\left(\left(t^{-1}\right)\right)}
\def\ft{\F_2}

\def\fttt{\ft(t)}
\def\ftttt{\ft\left(\left(t^{-1}\right)\right)}

\def\bbn{\mathbb{N}}
\def\bbz{\mathbb{Z}}
\def\frakh{\mathfrak{H}}
\def\frakt{\mathfrak{T}}

\def\frakf{\mathfrak{F}}

\tikzset{
	splus/.pic = {
		\foreach \i in {0,4}{
			\draw (0,\i) -- (4,\i);
			\draw (\i,0)-- (\i,4);
		}
		\draw (4,0)-- (0,4);
		\fill[fill=black] (0,0)--(0,2)--(1,3)--(3,1)--(2,0);
	}
}

\tikzset{
	sminus/.pic = {
		\foreach \i in {0,4}{
			\draw (0,\i) -- (4,\i);
			\draw (\i,0)-- (\i,4);
		}
		\draw (4,0)-- (0,4);
		\fill[fill=black] (0,2)--(1,3)--(0,4);
		\fill[fill=black] (2,0)--(3,1)--(4,0);
	}
}

\tikzset{
	tilera/.pic = {
		
		\foreach \i in {0,4}{
			\draw (0,\i) -- (4,\i);
			\draw (\i,0)-- (\i,4);
		}
		\draw (2,0)-- (0,2);
		\draw (2,0)-- (4,2);
		\draw (2,4)-- (0,2);
		\draw (2,4)-- (4,2);
		\fill[fill=black] (0,0)--(1,1)--(0,2);
		\fill[fill=black] (2,4)--(1,3)--(0,4);
		\fill[fill=black] (2,0)--(3,1)--(4,0);
		\fill[fill=black] (4,2)--(3,3)--(4,4);
	}
}

\tikzset{
	tilerb/.pic = {
		\foreach \i in {0,4}{
			\draw (0,\i) -- (4,\i);
			\draw (\i,0)-- (\i,4);
		}
		\draw (2,0)-- (0,2);
		\draw (2,0)-- (4,2);
		\draw (2,4)-- (0,2);
		\draw (2,4)-- (4,2);
		\fill[fill=black] (0,0)--(1,1)--(2,0);
		\fill[fill=black] (0,2)--(1,3)--(0,4);
		\fill[fill=black] (4,2)--(3,1)--(4,0);
		\fill[fill=black] (2,4)--(3,3)--(4,4);
	}
}

\tikzset{
	tilec/.pic = {
		\foreach \i in {0,4}{
			\draw (0,\i) -- (4,\i);
			\draw (\i,0)-- (\i,4);
		}
		\fill (0,0) circle (0.3);
	}
}

\begin{document}
\title{Escape of mass of the Thue-Morse sequence}

\author{Erez Nesharim}
\email{nesharim@technion.ac.il}
\address{Department of Mathematics, Technion - Israel Institute of Technology, Haifa, Israel}

\author{Uri Shapira}
\email{ushapira@technion.ac.il}
\address{Department of Mathematics, Technion - Israel Institute of Technology, Haifa, Israel}

\author{Noy Soffer Aranov}
\email{noysofferaranov@math.utah.edu}
\address{Department of Mathematics, University of Utah, Salt Lake City, Utah, USA}

\maketitle
\begin{abstract}
	Every Laurent series in $\fqttt$ has a continued fraction expansion whose partial quotients are polynomials. De Mathan and Teuli\'e proved that the degrees of the partial quotients of the left-shifts of every quadratic Laurent series are unbounded. Shapira and Paulin and Kemarsky improved this by showing that certain sequences of probability measures on the space of lattices in the plane $\fqttt^2$ exhibit positive escape of mass and conjectured that this escape is full -- that is, that these probability measures converge to zero. We disprove this conjecture by analysing in detail the case of the Laurent series over $\ft$ whose sequence of coefficients is the Thue-Morse sequence. The proof relies on 
	the discovery of explicit symmetries in its number wall.
\end{abstract}
	
\paragraph{}
\begin{center}
	\emph{Dedicated to Gozal Nesharim on his \nth{75} birthday}
\end{center}



\section{Introduction}
In this paper we study in detail a certain tiling of $\bbn^2$ by $\bF_2$ that is obtained by simple geometric means from 
the \emph{number wall} of the Thue-Morse sequence and establish its automaticity in an explicit manner. 
This tiling is tightly related to questions in Diophantine approximation 
of quadratic Laurent series in the local field $\ftttt$ and to the properties of an orbit of the 
full diagonal group in $\PGL_2\left(\ftttt\times \bF_2((t))\right)$ acting on the homogeneous space 
\begin{equation}\label{eq:homogeneous}
	\PGL_2\left(\ftttt\times \bF_2((t))\right)/\PGL_2\left(\bF_2\br{t,t^{-1}}\right).
\end{equation}
Our analysis of the tiling allows us to disprove certain statements in the Diophantine side that were widely believed to hold and, in that respect, are surprising. Consequently, this paper may be of interest to three groups of mathematicians with different mathematical backgrounds: automatic tilings, Diophantine approximation, and homogeneous dynamics. 

We split the introduction into two parts. We begin by explaining briefly the implication of the analysis carried out to Diophantine approximation and homogeneous dynamics. We then move to defining the necessary terminology in the theory of tilings and stating our main result.

\subsection{Partial escape of mass}

Let $q$ be a prime power. Every Laurent series $\theta\in\fqttt$ has a unique continued fraction expansion 
$$
\theta=a_0+\cfrac{1}{a_1+\cfrac{1}{a_2+\cdots\vphantom{\cfrac{1}{1}}}}\eqdef[a_0;a_1,a_2,\dots]\,,
$$
where $a_k=a_k(\theta)\in\fqt$ are polynomials for every $k \geq 0$, and
\begin{equation}\label{eq:dk}
	d_k = d_k(\theta) \defeq \deg\left(a_k\right)
\end{equation} 
for every $k \geq 1$ satisfies $d_k \geq 1$. It also has a unique base-$t$ expansion
$$
\theta = \sum_{n=-\infty}^\infty b_n t^{-n}\,,
$$
where $b_n=b_n(\theta)\in\fq$ for every $n\in\bbz$, and only finitely many of the coefficients of positive powers of $t$ are nonzero. Multiplication by $t$ changes the base-$t$ expansion by a left-shift. Therefore, for every $j,n\in\bbz$ the base-$t$ expansion of $t^j\theta$ satisfies \begin{equation}\label{eq:base-t-expansion}
	b_n\left(t^j\theta\right) = b_{n+j}\left(\theta\right).
\end{equation} 
It is natural to ask how $a_k\left(t^j\theta\right)$ vary as $k,j\rightarrow \infty$.

A very well studied scenario is when $\theta$ is quadratic over $\fqtt$.
De Mathan and Teuli\'e proved \cite[Theorem 4.5]{dMT} that
\begin{equation*}
    \limsup_{k,j\rightarrow\infty}d_k\left(t^j\theta\right)=\infty\,.
\end{equation*}
Paulin and Shapira \cite{PS}, relying on an earlier work with Kemarsky \cite{KPS},
improved this result by showing that, on average, a positive proportion of the degrees are unbounded, in the following sense:

\begin{definition}\label{def:escape}
	For any $\theta\in\fqttt$, $d>0$ and $k\geq1$, the \emph{mass of $\theta$ which lies above $d$ up to the $k$\textsuperscript{th} partial quotient} is 
	\begin{equation}\label{eq:escape}
		\es_\theta(d,k) \defeq \frac{\sum_{k'=1}^{k}\max\left\{d_{k'}-d,0\right\}}
		{\sum_{k'=1}^{k}d_{k'}}\,.
	\end{equation}
\end{definition}
\noindent
Their result \cite[Theorem 1.1]{PS} implies that if $q$ is not a power of $2$ and $\theta$ is quadratic over $\fqt$, then 
\begin{equation}\label{eq:uniform_escape}
\lim_{d\to\infty}\liminf_{j\rightarrow\infty}\lim_{k\to\infty}\es_{t^j\theta}(d,k) > 0\,.
\end{equation}

From the homogeneous dynamics point of view which is described in \cite{KPS} (and the analogous discussion in \cite{AkaShapira} in the real setup), every quadratic $\al\in \bF_q((t^{-1}))$ corresponds to a probability measure $\nu_\al$ on the homogeneous space of $\bF_q[t]$-\emph{lattices} $X=\PGL_2(\bF_q((t^{-1})))/\PGL_2(\bF_q[t])$. The meaning of the limit in~\eqref{eq:uniform_escape} being equal to $c$ is that any weak* accumulation point $\nu$ of the sequence of measures $\left\{\nu_{t^j\theta}\right\}_{j=1}^\infty$ satisfies $\nu(X)<1-c$. From this point of view, this paper is focused on the understanding of the measures $\left\{\nu_{t^j\theta}\right\}_{j=1}^\infty$. 

Our analysis will allow the computation of the limit in \eqref{eq:uniform_escape} for a particular example of $\theta$ which is of independent interest. Let $\tilea$ and $\tileb$ be two symbols and define a substitution $\mu:\{\tilea,\tileb\}\to\{\tilea,\tileb\}^2$ by \begin{align}\label{eq:tm-substitution}
	\mu(\tilea)&=\begin{array}{cc} \tilea&\tileb\end{array},\quad 
	&\mu(\tileb)&= \begin{array}{cc} \tileb&\tilea\end{array}.
\end{align} 
Consider the sequence of finite strings of $\tilea$'s and $\tileb$'s obtained by starting with the string $\tilea$ and applying $\mu$ iteratively to it:
$$\tilea\overset{\mu}{\lra}\tilea\tileb\overset{\mu}{\lra}\tilea\tileb\tileb\tilea\overset{\mu}{\lra}\tilea\tileb\tileb\tilea\tileb\tilea\tilea\tileb\overset{\mu}{\lra}\tilea\tileb\tileb\tilea\tileb\tilea\tilea\tileb\tileb\tilea\tilea\tileb\tilea\tileb\tileb\tilea\dots$$
The reader can easily convince themself by induction that for any $l\ge l'\ge 0$, the first $2^{l'}$ symbols in the sequence $\mu^l(\tilea)$ do not depend on $l$. This means that there is a well defined limit of this process in $\set{\tilea,\tileb}^\bN$.
This limit, which we denote by $\mu^\infty(\tilea)$,
is 
the \emph{Thue-Morse sequence}. 
Define a coding $\tau:\{\tilea,\tileb\}\to\{0,1\}$ by 
\begin{align}\label{eq:tm-coding}
	\tau(\tilea)&=0\,,\quad 
	&\tau(\tileb)&= 1\,.
\end{align}
The map $\tau$ extends naturally to strings of $\tilea$'s and $\tileb$'s. Denote 
\begin{equation}\label{eq: Thue-Morse}
\tm = \{\tm_n\}_{n=1}^\infty \defeq \tau \left(\mu^{\infty}(\tilea)\right)\,,
\end{equation}
and use the same letter for the corresponding Laurent series in $\ftttt$
$$
\tm = \sum_{n=1}^\infty\tm_nt^{-n}\,.
$$
It is 
well-known that
$\tm$ is quadratic over $\fttt$ (see \cite{AllSha}).
Our main application is
the following result, which disproves \cite[Conjecture 6]{KPS}, where it is predicted that the limit in~\eqref{eq:main} is $1$:
\begin{theorem}\label{thm:main}
	\begin{equation}\label{eq:main}
	\lim_{d\to\infty}\liminf_{j\rightarrow\infty}\lim_{k\to\infty}\es_{t^j\tm}(d,k) = \sfrac{2}{3}\,.
	\end{equation}
\end{theorem}
\noindent

\subsection{Tilings and diagonally-aligned number walls}
We now introduce some terminology in order to state our main Theorem~\ref{thm:main2}, which 
implies Theorem~\ref{thm:main}. We refer the reader to \cite[Section 4]{ALN} for a more thorough exposition of this terminology. 

Let $\Sig$ be a finite set and refer to its elements as \emph{tiles}. A 
tiling of a subset $S\sub \bZ^2$ by $\Sig$
is simply a map $T:S\to \Sig$. We think of the tiling as a planar array of tiles where the tile $T(m,n)$ is located at position $(m,n)$. Note that the way we arrange $\bZ^2$ in the plane is \textbf{not} the standard way but the one consistent with matrix notation. That is, as the row index $m$ increases the position advances \emph{downwards}, and as the column index $n$ grows the position advances \emph{rightwards}.

\ignore{
The study of the Thue-Morse sequence has a long history. Thue \cite{Thue1912} showed that for every $j\geq0$, $l>0$ and $u\in\{0,1\}^l$ there exists $1\leq n\leq 2l+1$ such that $\tm_{j+n}\neq u_{n \mod l}$ (see \cite{Berstel} for an exposition of Thue's work). In other words, the sequence $\tm$ has the property that all its left-shifts are poorly approximable by periodic sequences. Hence, it is natural to study approximations of $\left\{t^j\tm\right\}_{j=0}^\infty$ by rational functions in $\fttt$.
}

Let  
\begin{equation}\label{eq:no-polynomial-part}
	\theta = \sum_{n=1}^\infty b_n t^{-n}\in \fqttt.
\end{equation} 
It is well known 
that many of the Diophantine properties of the Laurent series $\theta$ are encoded in the tiling
\begin{equation}\label{eq:tm_hankel_array}
	\frakh(\theta) \defeq\left\{\frakh_{m,n}(\theta)\right\}_{m=1,n=1}^\infty,
\end{equation}
where for any $m,n\geq1$,
\begin{equation}\label{eq:tm_hankel}
	\frakh_{m,n}(\theta) \defeq \det
	\begin{pmatrix}
		b_{n}&b_{n+1}&\cdots &b_{n+m-1}\\
		b_{n+1}&\iddots&\iddots&b_{n+m}\\
		\vdots&\iddots&\iddots&\vdots\\
		b_{n+m-1}&b_{n+m}&\cdots&b_{n+2m-2}
	\end{pmatrix}
\end{equation}
is a determinant of a Hankel matrix which corresponds to the sequence $\{b_n\}_{n=1}^\infty$, computed in $\bF_q$.
More explicitly, for any $k\geq1$, denote 
\begin{equation}\label{eq:diophantine}
i_k=i_k(\theta)\defeq\sum_{k'=1}^kd_{k'}\,,
\end{equation}
where $\left\{d_k\right\}_{k=1}^\infty$ are the degrees of the partial quotients of $\theta$ defined in \eqref{eq:dk}.
The integer $i_k$ is the degree of the denominator of the $k$\textsuperscript{th} convergent of $\theta$. A standard fact in the theory of Diophantine approximation over function fields \cite[Setion 9]{dT04} is that
\begin{equation}\label{eq:last}
	\begin{aligned}
		\frakh_{i,1}(\theta)\neq0\ 
		&\iff
		\min_{M,N\in\fqt,\;N\neq0,\;\deg N<i} | N\theta - M |=q^{-i}  \\
		&\iff
		i=i_k(\theta) \text{ for some }k\geq1\,,
	\end{aligned}
\end{equation}
where the first equivalence follows by rewriting $| N\theta - M |=q^{-i}$ as a system of linear equations. By \eqref{eq:base-t-expansion} and \eqref{eq:tm_hankel}, the same holds more generally for every $j\geq0$ with $\frakh_{i,j+1}(\theta)$ and $t^{j}\theta$ instead of $\frakh_{i,1}(\theta)$ and $\theta$. It is evident from this and from Definition~\ref{def:escape} that in order to understand the approximations of $\left\{t^j\theta\right\}_{j=0}^\infty$ by rational functions in $\bF_q(t)$ in general and questions about escape of mass in particular, one can study the patterns of zeros in $\frakh(\theta)$.

\begin{figure}[ht]
	\centering
	\includegraphics[width=1\linewidth]{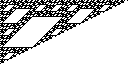}
	\caption{Hankel determinants of the Thue-Morse sequence $\frakh_{m,n}(\tm)$ for every $m,n\geq1$ such that $2m+n-2\leq128$, where the $0$'s are white and the $1$'s are black.}
	\label{fig:thm_hankel}
\end{figure}

For the Thue-Morse sequence over $\ft$, it is visually evident that $\frakh(\tm)$ has some self-similarity (see Figure \ref{fig:thm_hankel}). In their celebrated paper \cite{APWW}, Allouche, Peyri\`ere, Wen and Wen formalise this by proving that $\frakh(\tm)$ is \emph{automatic}, which loosely speaking means that it is obtained as a limit of a substitution and coding process similar to the one used above in the definition of $\tm$ (see~\cite[Section 14]{AllSha} for a treatment of $2$-dimensional automatic tilings).
However, in \cite{APWW} no explicit substitution and coding that generate $\frakh(\tm)$ are provided. This paper fills this gap. Apart from its importance to the subject of automatic tilings, this explicit structure 
is precisely what allows us to deduce Theorem~\ref{thm:main} in \S\ref{sec:uniform-escape-of-mass}.
Our main discovery is that the symmetries and structure of the tiling $\frakh(\tm)$, although still quite complex, become reasonably approachable 
after applying a certain linear 
map to the plane.
In order to do that, for any $\theta$ as in \eqref{eq:no-polynomial-part} we first extend the tiling $\frakh(\theta)$ to nonpositive $m$'s by setting $$
\frakh_{m,n}(\theta) = 
\begin{cases}
	1 & \textrm{for $m=0$ and all $n\geq1$}\\
	0 & \textrm{for all $m<0$ and $n\geq2|m|+1$}
\end{cases}.
$$
We then define 
\begin{equation}\label{eq: frakf tile}
\frakf_{m,n}(\theta)\defeq\begin{cases}
	 (-1)^{\tfrac{(m-n-1)(m-n+1)}{2}}\frakh_{\frac{m-n+1}{2},n}(\theta)&m\not\equiv n(2)\\
	0&m\equiv n(2)
\end{cases}
\end{equation}
for every $m\geq0$ and $n\geq1$. 
\ignore{Both the introduction of the new $0$'s in coordinates with different parities and 
the vertical shearing are natural from various perspectives which we do not address at the moment (\red{see Remark \ref{}}).}
The array
\begin{equation}\label{eq:tm_flat_array}
	\frakf(\theta)\defeq\left\{\frakf_{m,n}(\theta)\right\}_{m=0,n=1}^\infty\,
\end{equation}
is an example of a tiling we call \emph{diagonally-aligned number wall} (see Section \ref{sec:frame-relations} for a definition).
The main discovery in this paper is that $\frakf(\alpha)$ has a remarkable structure which turns out to be simple enough in order to be explicitly described (see Figure \ref{fig:thm_flat}).

\begin{figure}[ht]
	\centering
	\includegraphics[width=1\linewidth]{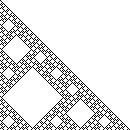}
	\caption{The (diagonally-aligned) number wall of the Thue-Morse sequence $\frakf_{m,n}(\tm)$ for $0\leq m\leq129$ and $1\leq n\leq130$, where the $0$'s are white and the $1$'s are black.}
	\label{fig:thm_flat}
\end{figure}

Our main result Theorem~\ref{thm:main2early} below identifies the tiling $\frakf(\tm)$ as an explicit coding of a substitution tiling. The definition of the substitution and coding requires some preparations which we now turn our attention to.

\subsection{The main theorem}
Our goal now is to define the necessary objects needed to state our main result. Unfortunately, we need to introduce quite a bit of structure in order for things to be understandable. Our plan is as follows:
\begin{itemize}
\item We define a set $\Sig$ of $15$ \emph{basic tiles} and a group $G$ of $16$ elements that acts on $\Sig$. 
\item We then extend the action of $G$ to an action on square matrices $\Mat_{2^l}(\Sig)$ for any $l\ge 0$, that involves both the $G$-action on the entries as well as a geometric action as isometries of the matrix thought of as a square.
\item We then use this $G$-action to define a $2$-substitution $\sig$ on $\Sig$ that \emph{commutes} with the $G$-actions defined. By choosing an appropriate tile $\splus{0}\in\Sig$, we obtain a $2$-dimensional substitution tiling $\sig^\infty(\splus{0})$ 
with remarkable symmetries which are captured by the $G$-actions. 
\item We then define a $4$-coding $\ka_1: \Sig\lra \Mat_4(\bF_2)$ and obtain the $2$-dimensional tiling $\ka_1\left(\sig^\infty(\splus{0})\right)$ of $\bN^2$ by $\ft$.
\end{itemize} 
The following theorem is the main theorem of this paper. It will be restated again in the end of the current subsection after we carry out the plan laid above.
\begin{theorem}\label{thm:main2early}\label{thm:main2}
Let $\tm$ be the Thue-Morse sequence as in~\eqref{eq: Thue-Morse}. 
The substitution $\sigma$ and coding $\kappa_1$ generate $\frakf(\tm)$. More precisely,
$$
\left\{\frakf_{m,n}(\tm)\right\}_{m=1,n=1}^\infty=\kappa_1\left(\sigma^\infty(\splus{0})\right).
$$
\end{theorem}
 
We now execute the above plan that will allow the reader to understand the automatic tiling $\ka_1\left(\sig^\infty(\splus{0})\right)$ featuring in Theorem~\ref{thm:main2early}.
\subsubsection{The basic tiles and their group of symmetries}
The simplest tile in $\Sigma$ is denoted by $\tilo$, and it is graphically depicted as:

$$\tilo = 
\scalebox{0.3}{
	\begin{tikzpicture}
		\foreach \i in {0,4}{
			\draw (0,\i) -- (4,\i);
			\draw (\i,0)-- (\i,4);
		}
	\end{tikzpicture}
}
$$

The dihedral group $D$ of order $8$ (the isometry group of a square) acts naturally on the basic tiles in $\Sig$. Denote by $\rho$ the isometry of the square that rotates it by $90\degree$ clockwise and by $\eta$ the reflection of the square along the anti-diagonal. As a group $D$ has the following presentation 
$$D = \idist{\rho,\eta \,|\, \rho^4, \eta^2, \rho\eta\rho\eta}.$$
The basic tile $\tilo$ is of course a fixed point for the action of $D$.
The next collection of basic tiles consists of tiles which are denoted by $\tilec{i}$ for $i=0,\ldots,3$. Their graphic depiction is as follows:

$$\tilec{0} = 
\scalebox{0.3}{
	\begin{tikzpicture}
		\foreach \i in {0,4}{
			\draw (0,\i) -- (4,\i);
			\draw (\i,0)-- (\i,4);
		}
		\fill (0,0) circle (0.3);
		\end{tikzpicture}
}\,,\quad
%
%
\tilec{1} = 
\scalebox{0.3}{
	\begin{tikzpicture}
		\foreach \i in {0,4}{
			\draw (0,\i) -- (4,\i);
			\draw (\i,0)-- (\i,4);
		}
		\fill (0,4) circle (0.3);
	\end{tikzpicture}
}\,,\quad
%
%
\tilec{2} = 
\scalebox{0.3}{
	\begin{tikzpicture}
		\foreach \i in {0,4}{
			\draw (0,\i) -- (4,\i);
			\draw (\i,0)-- (\i,4);
		}
		\fill (4,4) circle (0.3);
	\end{tikzpicture}
}\,,\quad
%
%
\tilec{3} = 
\scalebox{0.3}{
	\begin{tikzpicture}
		\foreach \i in {0,4}{
			\draw (0,\i) -- (4,\i);
			\draw (\i,0)-- (\i,4);
		}
		\fill (4,0) circle (0.3);
	\end{tikzpicture}
}\,.
%
%
$$
It should be clear that $\set{\tilec{i}:0\le i\le 3}$ is an orbit of $D$. In fact, for every $0\leq i\leq 3$ we have $\tilec{i} = \rho^i\tilec{0}$, while $\tilec{0},\tilec{2}$ are fixed points for $\eta$ and $\eta\tilec{1} = \tilec{3}$. 
The next two tiles are denoted by $\tiler{a}$ and $\tiler{b}$, and their graphic depiction is as follows:
$$
\tiler{a} = 
\scalebox{0.3}{
	\begin{tikzpicture}
		\foreach \i in {0,4}{
			\draw (0,\i) -- (4,\i);
			\draw (\i,0)-- (\i,4);
		}
		\draw (2,0)-- (0,2);
		\draw (2,0)-- (4,2);
		\draw (2,4)-- (0,2);
		\draw (2,4)-- (4,2);
		\fill[fill=black] (0,0)--(1,1)--(0,2);
		\fill[fill=black] (2,4)--(1,3)--(0,4);
		\fill[fill=black] (2,0)--(3,1)--(4,0);
		\fill[fill=black] (4,2)--(3,3)--(4,4);
	\end{tikzpicture}
}\,,\quad
%
\tiler{b} = 
\scalebox{0.3}{
	\begin{tikzpicture}
		\foreach \i in {0,4}{
			\draw (0,\i) -- (4,\i);
			\draw (\i,0)-- (\i,4);
		}
		\draw (2,0)-- (0,2);
		\draw (2,0)-- (4,2);
		\draw (2,4)-- (0,2);
		\draw (2,4)-- (4,2);
		\fill[fill=black] (0,0)--(1,1)--(2,0);
		\fill[fill=black] (0,2)--(1,3)--(0,4);
		\fill[fill=black] (4,2)--(3,1)--(4,0);
		\fill[fill=black] (2,4)--(3,3)--(4,4);
	\end{tikzpicture}
}\,.\quad
$$
It should be clear that $\set{\tiler{a},\tiler{b}}$ is an orbit of $D$. In fact, the tiles $\tiler{a}$ and $\tiler{b}$ are fixed points for $\rho$, while $\eta\tiler{a}=\tiler{b}$. 
The next collection of tiles in $\Sig$ consists of tiles which are denoted by $\splus{i}$ for $i=0,\dots,3$. Their graphic depiction is as follows:
$$
\splus{0} = 
\scalebox{0.3}{
	\begin{tikzpicture}
		\foreach \i in {0,4}{
			\draw (0,\i) -- (4,\i);
			\draw (\i,0)-- (\i,4);
		}
		\draw (4,0)-- (0,4);
		\fill[fill=black] (0,0)--(0,2)--(1,3)--(3,1)--(2,0);
	\end{tikzpicture}
}\,,\quad
\splus{1} = 
\scalebox{0.3}{
	\begin{tikzpicture}
		\foreach \i in {0,4}{
			\draw (0,\i) -- (4,\i);
			\draw (\i,0)-- (\i,4);
		}
		\draw (0,0)-- (4,4);
		\fill[fill=black] (0,4)--(2,4)--(3,3)--(1,1)--(0,2);
	\end{tikzpicture}
}\,,\quad
\splus{2} = 
\scalebox{0.3}{
	\begin{tikzpicture}
		\foreach \i in {0,4}{
			\draw (0,\i) -- (4,\i);
			\draw (\i,0)-- (\i,4);
		}
		\draw (4,0)-- (0,4);
		\fill[fill=black] (4,4)--(2,4)--(1,3)--(3,1)--(4,2);
	\end{tikzpicture}
}\,,\quad
\splus{3} = 
\scalebox{0.3}{
	\begin{tikzpicture}
		\foreach \i in {0,4}{
			\draw (0,\i) -- (4,\i);
			\draw (\i,0)-- (\i,4);
		}
		\draw (0,0)-- (4,4);
		\fill[fill=black] (4,0)--(2,0)--(1,1)--(3,3)--(4,2);
	\end{tikzpicture}
}\,.\quad
%
$$
It should be clear that $\set{\splus{i}:0\le i\le 3}$ is a $D$-orbit. In fact, for every $0\leq i\leq 3$ we have $\splus{i} = \rho^i\splus{0}$, the tiles $\splus{0}$ and $\splus{2}$ are fixed points for $\eta$, while $\eta\splus{1}=\splus{3}$. 

The last collection of basic tiles in $\Sig$ consists of tiles denoted by $\sminus{i}$ for $i=0,\dots,3$. Their graphic depiction is as follows:

$$
\sminus{0} = 
\scalebox{0.3}{
	\begin{tikzpicture}
		\foreach \i in {0,4}{
			\draw (0,\i) -- (4,\i);
			\draw (\i,0)-- (\i,4);
		}
		\draw (4,0)-- (0,4);
		\fill[fill=black] (0,2)--(1,3)--(0,4);
		\fill[fill=black] (2,0)--(3,1)--(4,0);
	\end{tikzpicture}
}\,,\quad
\sminus{1} = 
\scalebox{0.3}{
	\begin{tikzpicture}
		\foreach \i in {0,4}{
			\draw (0,\i) -- (4,\i);
			\draw (\i,0)-- (\i,4);
		}
		\draw (0,0)-- (4,4);
		\fill[fill=black] (0,0)--(1,1)--(0,2);
		\fill[fill=black] (2,4)--(3,3)--(4,4);
	\end{tikzpicture}
}\,,\quad
\sminus{2} = 
\scalebox{0.3}{
	\begin{tikzpicture}
		\foreach \i in {0,4}{
			\draw (0,\i) -- (4,\i);
			\draw (\i,0)-- (\i,4);
		}
		\draw (4,0)-- (0,4);
		\fill[fill=black] (2,4)--(1,3)--(0,4);
		\fill[fill=black] (4,2)--(3,1)--(4,0);
	\end{tikzpicture}
}\,,\quad
\sminus{3} = 
\scalebox{0.3}{
	\begin{tikzpicture}
		\foreach \i in {0,4}{
			\draw (0,\i) -- (4,\i);
			\draw (\i,0)-- (\i,4);
		}
		\draw (0,0)-- (4,4);
		\fill[fill=black] (0,0)--(1,1)--(2,0);
		\fill[fill=black] (4,2)--(3,3)--(4,4);
	\end{tikzpicture}
}\,.\quad
$$
It should be clear that $\set{\sminus{i}:0\le i\le 3}$ is a $D$-orbit. In fact, for every $0\leq i \leq3$ we have $\sminus{i} = \rho^i\sminus{0}$, the tiles $\sminus{0}$ and $\sminus{2}$ are fixed points for $\eta$, while $\eta\sminus{1}=\sminus{3}$.

To summarise, we have defined the set of basic tiles 
$$\Sig = \set{\splus{i}, \sminus{i}, \tilec{i}:0\le i\le 3}\cup\set{\tiler{a},\tiler{b},\tilo}$$
which consists of $15$ tiles, on which the dihedral group $D$ acts and partition it into five orbits.

Next, we introduce an extra symmetry on $\Sig$ that flips between $\splus{i}$ and $\sminus{i}$ for every $0\leq i\leq3$ and also flips between $\tiler{a}$ and $\tiler{b}$. This extra symmetry is crucial for our analysis. 
Let $$\bbz/2\bbz=\idist{\iota | \iota^2}$$ be a group of order two. 
\begin{definition}\label{def: iota}
The action of $\bbz/2\bbz$ on $\Sigma$ is defined by the following involution: 
\begin{gather*}
	\iota\tilo =\tilo\,,\\
	\begin{align*}
		\iota\tilec{i}&=\tilec{i}\,, &\quad& \textrm{for all $0\le i\le 3$}\,,
	\end{align*}\\
	\begin{align*}
		\iota \tiler{a} &= \tiler{b}\,, &\textrm{ and }& &\iota\tiler{b}&=\tiler{a}\,,&\quad&\\
		\iota \splus{i} &= \sminus{i}\,, &\textrm{ and }& &\iota\sminus{i}&=\splus{i}\,, &\quad& \textrm{for all $0\le i\le 3$}\,.
	\end{align*}
\end{gather*}
\end{definition} 
The action of $\bbz/2\bbz$ on $\Sig$ commutes with the action of $D$. Therefore, this gives an action of the group
$$G: = D\times\bbz/2\bbz$$
on $\Sig$, where $D$ and $\bbz/2\bbz$ are viewed as subgroups of $G$. The decomposition of $\Sig$ into $G$-orbits is almost identical to the decomposition into $D$-orbits, just that now, the two $D$-orbits $\set{\splus{i}}_{i=0}^3$ and  $\set{\sminus{i}}_{i=0}^3$ become one $G$-orbit of $8$ elements. We have
\begin{equation}\label{eq: Sigma splits to G-orbits}
\Sig = G\splus{0}\cup G\tilec{0}\cup G\tiler{a}\cup G\tilo\,.
\end{equation}

\subsubsection{Actions on matrices}
\newcommand{\tgeo}{T^{\on{geo}}}
\newcommand{\tinn}{T^{\on{ent}}}
When one looks at the tiling $\frakf(\tm)$ in different scales and stares at it enough, the graphic depictions of 
the tiles in $\Sig$ emerge (see Figure \ref{fig:thm_flat_large}). 

\begin{figure}[ht]
	\centering
	\includegraphics[width=1\linewidth]{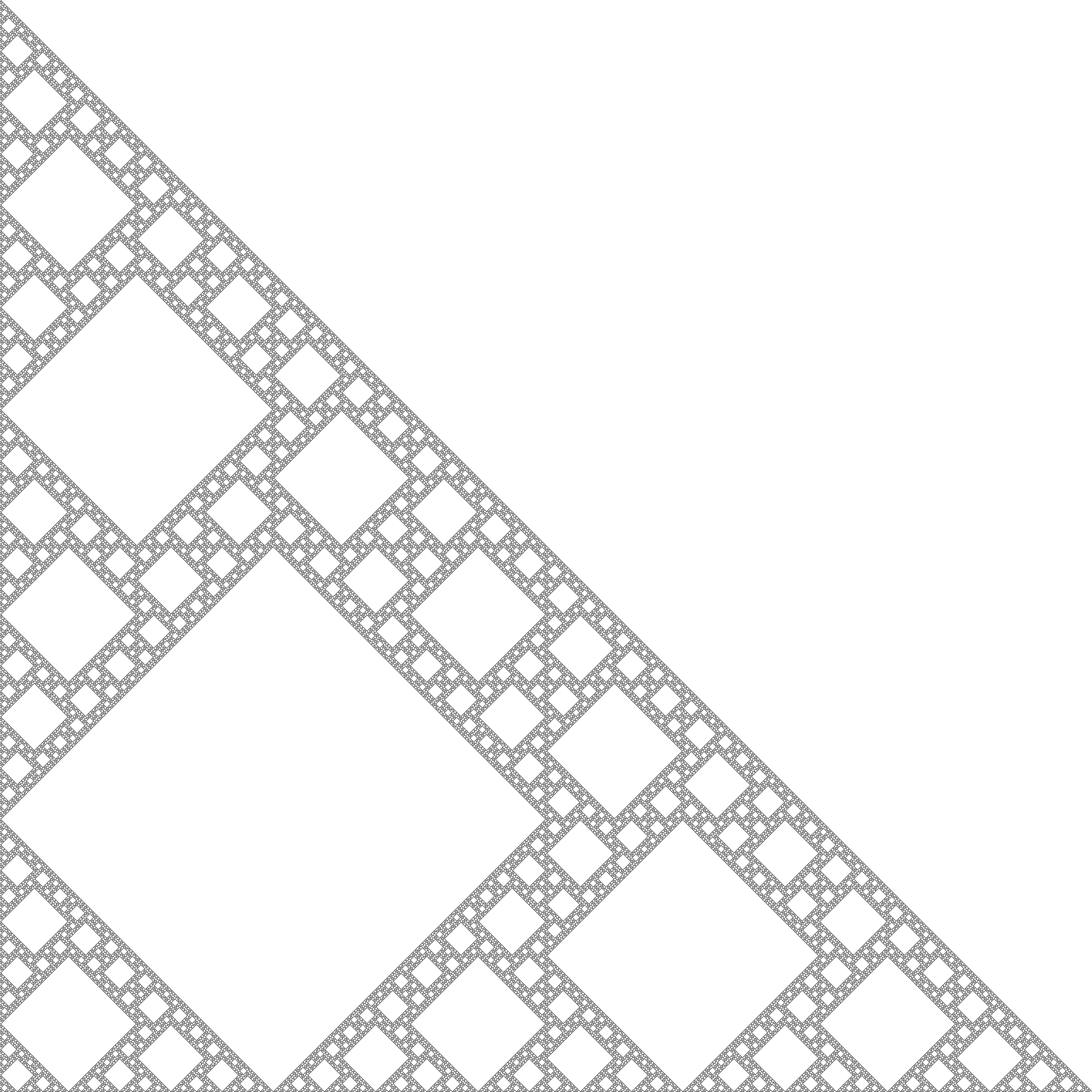}
	\caption{The (diagonally-aligned) number wall of the Thue-Morse sequence $\frakf_{m,n}(\tm)$ for $0\leq m\leq2049$ and $1\leq n\leq2050$, where the $0$'s are white and the $1$'s are black.}
	\label{fig:thm_flat_large}
\end{figure}

In order to be able to exploit the $G$-action on these tiles for our cause, it is necessary to generalise this action to an action of $G$ on
square matrices over $\Sig$. There are two natural actions:
\begin{itemize}
\item For $g\in G$ and a matrix $M\in \Mat_n(\Sig)$, let $\tinn_g M$ denote the matrix obtained from $M$ by acting on each of its entries by $g$. It is clear that this defines a $G$-action on $\Mat_n(\Sig)$ for any $n\ge 1$.
\item For $g\in G$ and a matrix $M\in \Mat_n(\Sig)$, let $\tgeo_g M$ denote the matrix obtained from $M$ by acting on it by the isometry of the square $\pi(g)$, where 
$$\pi:G\to D$$ 
is the natural projection. For example, $\tgeo_\iota$ acts trivially, $\tgeo_\rho M$ rotates $M$ by $90\degree$ clockwise and $\tgeo_\eta M$ reflects $M$ along
the anti-diagonal.
It should be clear that this defines an action of $G$ on $\Mat_n(\Sig)$ for any $n\ge 1$.
\end{itemize}
These two actions commute in the sense that for every $n\geq1$ and $M\in\Mat_n(\Sig)$, every $g,h\in G$ satisfy
\begin{equation}\label{eq: commutation of tgeo and tinn}
\tinn_g\tgeo_h M = \tgeo_h\tinn_g M\,.
\end{equation}
In reality we will use a \emph{twisted} combination of these two actions. 
The twist is obtained from an automorphism $\psi:G\to G$, which we now introduce.
\begin{definition}
The function $\psi:G\to G$ is the unique group automorphism which is defined on the generators of $G$ by
\begin{equation}\label{eq:above}
	\psi(\rho) = \iota\rho\,,\quad\psi(\eta) = \iota\eta\,,\quad\psi(\iota) =\iota\,.
\end{equation}
\end{definition}
The reader can verify that there exists a unique automorphism $\psi$ of $G$ that satisfies \eqref{eq:above}. Such a verification amounts to showing that the group generated by $\iota\rho$ and  $\iota\eta$ is isomorphic to $D$ and that it does not contain $\iota$. 

We now define for any $l\ge 0$ the \emph{twisted action} of $G$ on $\Mat_{2^l}(\Sig)$. This action will extend the $G$-action described above on $\Sig$ which corresponds to the case $l = 0$. 

\begin{definition}\label{def: action on matrices}
The action of $G$ on $\Mat_{2^l}(\Sig)$ is as follows: for every $g\in G$ and $M\in\Mat_{2^l}(\Sig)$, the action of $g$ on $M$ is
\begin{equation}\label{eq: action on matrices}
g M \defeq \tgeo_g\tinn_{\psi^l(g)} M\,.
\end{equation}
In words, the action of $g$ on a matrix is defined by the geometric action of $\pi(g)$ as an isometry on the matrix $M$ thought of as a square and the action of $\psi^l(g)$ on each of the entries of the matrix $M$. 
\end{definition}
Note that the commutation \eqref{eq: commutation of tgeo and tinn} implies that this is indeed an action of $G$ on
$\Mat_{2^l}(\Sig)$ for any $l\ge 0$. Also note, that for $l=0$ we recover the action of $G$ on $\Sig$. 
\begin{example}
It is instructive to work out explicitly what happens for $2\times 2$ matrices ($l = 1$):
\begin{align}
\nonumber 
\rho\mat{a&b\\c&d}& =\mat{\psi( \rho) c&\psi( \rho) a\\\psi( \rho)d&\psi( \rho) b}  
= \mat{\iota \rho c&\iota \rho a\\\iota \rho d&\iota \rho b}, \\
\label{eq: acting with eta on 2 matrix}
\eta\mat{a&b\\c&d} &=  \mat{\psi( \eta) d&\psi( \eta) b\\\psi(\eta)c&\psi(\eta) a}
= \mat{\iota\eta d &\iota\eta b\\ \iota\eta c&\iota\eta a},\\
\nonumber \iota \mat{a&b\\c&d} &= \mat{\psi( \iota) a&\psi(  \iota) b\\\psi( \iota)c&\psi( \iota) d}
= \mat{\iota a &\iota b\\ \iota c&\iota d}.
\end{align}
\end{example}

\subsubsection{The $2$-substitution} 
We use the $G$-action on $\Sig$ and on $\Mat_2(\Sig)$ in order to define a $2$-substitution $\sigma:\Sig\to\Mat_2(\Sig)$, while having in mind the equivariance relation $$\sig(gs) = g\sig(s)$$ which we would like to hold. 
Recall that $\Sig$ splits into $4$ orbits of $G$ as in~\eqref{eq: Sigma splits to G-orbits}. The first step is
to define $\sig$ on the set of representatives 
$$\Sig_0:= \set{\splus{0},\tilec{0},\tiler{a},\tilo}.$$
\begin{definition}\label{def: the substitution sigma part 1}
The substitution $\sig$ is defined on $\Sig_0$ as follows:
\begin{align*}
\sig(\splus{0})& = \mat{\splus{0}& \tilec{0}\\ \tiler{a}&\sminus{0}}, &\quad 
\sig(\tilec{0}) &=  \mat{\tilo&\tilo \\ \tilec{0}&\tilo},\\[2pt]
\sig\left(\tiler{a}\right)  &= \mat{\sminus{1}& \splus{2} \\ \splus{0}&\sminus{3}}, &\quad
\sig(\tilo)   &= \mat{\tilo&\tilo \\ \tilo&\tilo}.
\end{align*}
\end{definition}
It is useful to have in mind the graphic depiction of the substitution, which can be generated by putting the depictions of each of the tiles adjacent to each other. For example, the image of $\splus{0}$ and $\tiler{a}$ are
\begin{align*}
\sig(\splus{0}) &=
\scalebox{0.3}{
	\begin{tikzpicture}
		\path  (0,0) pic{tilera};
		\path  (4,0) pic{sminus};
		\path  (0,4) pic{splus};
		\path  (4,4) pic{tilec};
	\end{tikzpicture}
}\,,\quad
&\sig\left(\tiler{a}\right)& = 
	\scalebox{0.3}{
		\begin{tikzpicture}
			\path  (0,0) pic{splus};
			\path[rotate=90,transform shape]  (0,-8) pic{sminus};
			\path[rotate=180,transform shape]  (-8,-8) pic{splus};
			\path[rotate=270,transform shape]  (-8,0) pic{sminus};
		\end{tikzpicture}
	}\,.
\end{align*}

The following lemma is essential for the extension of the definition of $\sig$ from $\Sig_0$ to all of $\Sig$ equivariantly.
\begin{lemma}	
\label{lem: stabilizer}
Every $s\in \Sig_0 = \set{\splus{0},\tilec{0},\tiler{a},\tilo}$
satisfies
\begin{equation}\label{eq: inclusion of stabilizers}
\on{Stab}_G(s) \leq \on{Stab}_G(\sig(s))\,.
\end{equation}
\end{lemma}
\begin{proof}
The proof follows by checking \eqref{eq: inclusion of stabilizers} for every $s\in \Sig_0$. Start with $s = \splus{0}$. Since $\av{G}=16$ and $\av{G\splus{0}}=8$, 
the stabilizer group $\on{Stab}_G(\splus{0})$ has $2$ elements. Since it contains the involution $\eta$, it must be equal to $\idist{\eta}$. Acting with $\eta$ on $\sig(\splus{0})$ using the twisted $G$-action, equation \eqref{eq: acting with eta on 2 matrix} gives
$$\eta\sig(\splus{0}) = \eta\mat{\splus{0}& \tilec{0}\\ \tiler{a}&\sminus{0}}
=\mat{\iota\eta \sminus{0}&\iota\eta \tilec{0}\\\iota\eta \tiler{a}&\iota \eta \splus{0}}=
\mat{\splus{0}& \tilec{0}\\ \tiler{a}& \sminus{0}},$$
so $\eta\in\on{Stab}_G(\sig(\splus{0}))$.

Proceed with $s = \tiler{a}$. The orbit $G\tiler{a}$ has $2$ elements and hence its stabilizer group is of order $8$. It clearly contains $\rho$ and $\iota\eta$ (since both $\iota$ and $\eta$ switch between $\tiler{a}$ and $\tiler{b}$). Since the group $\idist{\rho,\iota\eta}$ has $8$ elements, it is $\on{Stab}_G\left(\tiler{a}\right)$. Thus, in order to verify \eqref{eq: inclusion of stabilizers} it is enough to show that $\rho$ and $\iota\eta$ stabilize $\sig\left(\tiler{a}\right)$. 
Working with the definition of the twisted action on matrices (see \eqref{eq: acting with eta on 2 matrix}) gives
\begin{align*}
\rho\sig\left(\tiler{a}\right) 
&= 
\rho\mat{\sminus{1}& \splus{2} \\ \splus{0}&\sminus{3}} 
= 
\mat{ \psi(\rho) \splus{0}&  \psi(\rho)\sminus{1} \\  \psi(\rho)\sminus{3} & \psi(\rho) \splus{2}}
\\[1.5pt]
&= 
\mat{ \iota\rho \splus{0}& \iota\rho\sminus{1} \\ \iota\rho\sminus{3} & \iota\rho\splus{2}}  
= 
\mat{\sminus{1}& \splus{2} \\ \splus{0}&\sminus{3}}.
\end{align*}
Similarly, 
\begin{align*}
\eta\iota\sig\left(\tiler{a}\right) 
&= 
\eta\iota\mat{\sminus{1}& \splus{2} \\ \splus{0}&\sminus{3}} 
= 
\mat{ \psi(\eta\iota) \sminus{3}&  \psi(\eta\iota) \splus{2}  \\  \psi(\eta\iota)\splus{0}& \psi(\eta\iota) \sminus{1}}
\\[1.5pt]
&=
\mat{\eta \sminus{3}& \eta  \splus{2}  \\  \eta \splus{0}&\eta  \sminus{1}} 
=
\mat{\sminus{1}& \splus{2} \\ \splus{0}&\sminus{3}}.
\end{align*}

Next, equation \eqref{eq: inclusion of stabilizers} is verified for $s=\tilec{0}$. As for the previous cases, it is easy to see that since the 
orbit has $4$ elements, the stabilizer is of order $4$, so it must be $\idist{\eta,\iota}$. Therefore, it is enough to check that $\iota,\eta\in \on{Stab}_G(\sig(\tilec{0}))$. Indeed,
$$\eta\sig(\tilec{0}) = \eta \mat{\tilo&\tilo \\ \tilec{0}&\tilo} = 
\mat{\psi(\eta)\tilo&\psi(\eta)\tilo \\ \psi(\eta)\tilec{0}&\psi(\eta)\tilo}=
\mat{\eta\iota\tilo&\eta\iota\tilo \\ \eta\iota\tilec{0}&\eta\iota\tilo} =
\mat{\tilo&\tilo \\ \tilec{0}&\tilo}. $$
Similarly,
$$\iota\sig(\tilec{0}) = \iota \mat{\tilo&\tilo \\ \tilec{0}&\tilo} = 
\mat{\psi(\iota)\tilo&\psi(\iota)\tilo \\ \psi(\iota)\tilec{0}&\psi(\iota)\tilo}=
\mat{\iota\tilo&\iota\tilo \\ \iota\tilec{0}&\iota\tilo} =
 \mat{\tilo&\tilo \\ \tilec{0}&\tilo}. $$
Finally, the verification of \eqref{eq: inclusion of stabilizers} for $s=\tilo$ is left to the reader.
\end{proof}
Given Lemma~\ref{lem: stabilizer}, we extend the definition of $\sig$ to all of $\Sig$ as follows:
\begin{definition}\label{def: the substitution  sigma part 2}
For any $s\in \Sig$, write $s = gs_0$ for some $s_0\in \Sig_0$ and $g\in G$. 
Then
$$\sig(s) = \sig(gs_0): = g\sig(s_0)\,,$$
where $\sig(s_0)$ is as in Definition~\ref{def: the substitution sigma part 1}.
\end{definition}

This is well defined. Indeed, 
if $s = g_1s_1=g_2s_2$ for some $s_1,s_2\in \Sig_0$, then by construction $s_1=s_2$ and $g_2^{-1}g_1\in \on{Stab}_G(s_1)$. So, by Lemma~\ref{lem: stabilizer} $g_2^{-1}g_1\in \on{Stab}_G(\sig(s_1))$, which implies that $g_1\sig(s_1)=g_2\sig(s_2)$, as needed.

As hinted above, the commutation relation $\sig(gs) = g\sig(s)$ for $s\in \Sig$ will be a key feature in our ability to analyse this substitution. It is the seed that propagates and governs the symmetries in $\sig^\infty(\splus{0})$.
\begin{proposition}\label{prop: equivariance}
The $G$-actions on $\Sigma$ and on $\Mat_2(\Sig)$ commute with $\sigma$. More explicitly, every $g\in G$ and $s\in \Sig$ satisfy
\begin{equation}\label{eq:equivariance}
	\sig(gs) = g\sig(s)\,.
\end{equation}
\end{proposition}  
\begin{proof}
By Definition~\ref{def: the substitution  sigma part 2} this is almost a tautology, but there is still something to check: given $g\in G$ and $s\in \Sig$, write $s=g_0s_0$ for some $g_0\in G$ and $s_0\in \Sig_0$. Definition~\ref{def: the substitution  sigma part 2} gives
$$\sig(gs) = \sig(gg_0s_0) = gg_0\sig(s_0)\,.$$
On the other hand, by the same definition,
$$g\sig(s) = g\sig(g_0s_0) = gg_0\sig(s_0)\,,$$
which finishes the proof. 
\end{proof}
The readers are encouraged to calculate $\sig^l(\splus{0})$ for small values of $l$ by themselves or to look ahead at Example~\ref{example: first iterations} and Figure~\ref{fig:thm_infls}.

\subsubsection{The $4$-coding}
We now define a coding $\ka$ that replaces the tiles in $\Sig$ by matrices in $\Mat_5(\bF_2)$. The purpose of the coding is to generate a tiling over $\bF_2$. It does so by replacing each basic tile with a pixelated version of its graphic depiction.  Explicitly, let $\kappa:\Sigma\to\Mat_5(\ft)$ be the $5$-coding that is defined by
\begin{equation}\label{eq:kappa}
	\begin{aligned}
		\kappa(\tilo)&=
		\mat{
			0&0&0&0&0\\
			0&0&0&0&0\\
			0&0&0&0&0\\
			0&0&0&0&0\\
			0&0&0&0&0
		}
		,\quad
		&\kappa(\tilec{0})&=
		\mat{
			0&0&0&0&0\\
			0&0&0&0&0\\
			0&0&0&0&0\\
			1&0&0&0&0\\
			0&1&0&0&0
		}
		,\quad \\[3pt]
		\kappa\left(\tiler{a}\right)&=\mat{
			0&1&0&0&0\\
			0&0&1&0&1\\
			0&1&0&1&0\\
			1&0&1&0&0\\
			0&0&0&1&0
		}
		,\quad &\kappa\left(\tiler{b}\right)&=\mat{
			0&0&0&1&0\\
			1&0&1&0&0\\
			0&1&0&1&0\\
			0&0&1&0&1\\
			0&1&0&0&0
		}
		,\\[3pt]
		\kappa(\splus{0})&=\mat{
			0&1&0&0&0\\
			0&0&1&0&0\\
			0&1&0&1&0\\
			1&0&1&0&1\\
			0&1&0&0&0
		}
		,\quad &\kappa(\sminus{0})&=\mat{
			0&1&0&0&0\\
			1&0&1&0&0\\
			0&0&0&1&0\\
			0&0&0&0&1\\
			0&0&0&1&0
		}
		,
	\end{aligned}
\end{equation}
and so that the equation
\begin{equation}\label{eq:kappa-equivariance}
	\ka(\rho s) = \rho\ka(s)
\end{equation} 
holds for any $s\in \Sig$. Here the action of $\rho$ on the right hand side is the natural action as an isometry of the square (rotating by $90\degree$ clockwise).
It is useful to have in mind an image of the coding, which are black and white squares arranged in a matrix of size $5$ where each $0$ and $1$ is replaced by a white and black square, respectively. 
\begin{gather*}
\ka(\tilo) = 
\scalebox{0.22}{
	\begin{tikzpicture}
		\foreach \i in {0,...,5}{
			\draw (0,\i) -- (5,\i);
			\draw (\i,0)-- (\i,5);
		}
	\end{tikzpicture}
}\,,
\\[1.5pt]
\begin{align*}
\ka(\tilec{0}) = 
\scalebox{0.22}{
	\begin{tikzpicture}
		\foreach \i in {0,...,5}{
			\draw (0,\i) -- (5,\i);
			\draw (\i,0)-- (\i,5);
		}
		%
		\draw[fill] (0,1) rectangle (1,2);
		\draw[fill] (1,0) rectangle (2,1);
	\end{tikzpicture}
}\,,\quad
\ka(\tilec{1})  &=  
\scalebox{0.22}{
	\begin{tikzpicture}
	  \foreach \i in {0,...,5}{
		  \draw (0,\i) -- (5,\i);
		  \draw (\i,0)-- (\i,5);
		  }
	  %
	  \draw[fill] (0,3) rectangle (1,4);
	  \draw[fill] (1,4) rectangle (2,5);
	\end{tikzpicture}
	}\,,\quad
&\ka(\tilec{2}) &= 
\scalebox{0.22}{
	\begin{tikzpicture}
	  \foreach \i in {0,...,5}{
		  \draw (0,\i) -- (5,\i);
		  \draw (\i,0)-- (\i,5);
		  }
	  %
	  \draw[fill] (3,4) rectangle (4,5);
	  \draw[fill] (4,3) rectangle (5,4);
	\end{tikzpicture}
	}\,,\quad
\ka(\tilec{3}) =
\scalebox{0.22}{
	\begin{tikzpicture}
	  \foreach \i in {0,...,5}{
		  \draw (0,\i) -- (5,\i);
		  \draw (\i,0)-- (\i,5);
		  }
	  %
	  \draw[fill] (4,1) rectangle (5,2);
	  \draw[fill] (3,0) rectangle (4,1);
	\end{tikzpicture}
	}\,,\quad
\\[1.5pt]
\ka\left(\tiler{a}\right) &= 
\scalebox{0.22}{
	\begin{tikzpicture}
	  \foreach \i in {0,...,5}{
		  \draw (0,\i) -- (5,\i);
		  \draw (\i,0)-- (\i,5);
		  }
	  %
	\draw[fill] (1,2) rectangle (2,3);
	\draw[fill] (2,3) rectangle (3,4);
	\draw[fill] (2,1) rectangle (3,2);
	\draw[fill] (3,2) rectangle (4,3);
	\draw[fill] (3,0) rectangle (4,1);
	\draw[fill] (4,3) rectangle (5,4);
	\draw[fill] (1,4) rectangle (2,5);
	\draw[fill] (0,1) rectangle (1,2);
	\end{tikzpicture}
	}\,,\quad
&\ka\left(\tiler{b}\right) &= 
\scalebox{0.22}{
	\begin{tikzpicture}
	  \foreach \i in {0,...,5}{
		  \draw (0,\i) -- (5,\i);
		  \draw (\i,0)-- (\i,5);
		  }
	  %
	\draw[fill] (1,2) rectangle (2,3);
	\draw[fill] (2,3) rectangle (3,4);
	\draw[fill] (2,1) rectangle (3,2);
	\draw[fill] (3,2) rectangle (4,3);
	\draw[fill] (0,3) rectangle (1,4);
	\draw[fill] (3,4) rectangle (4,5);
	\draw[fill] (4,1) rectangle (5,2);
	\draw[fill] (1,0) rectangle (2,1);
	\end{tikzpicture}
	}\,,\quad
\\[1.5pt]
\ka(\splus{0}) = 
\scalebox{0.22}{
	\begin{tikzpicture}
	  \foreach \i in {0,...,5}{
		  \draw (0,\i) -- (5,\i);
		  \draw (\i,0)-- (\i,5);
		  }
	  %
	  \foreach \i in {1,...,4}{
		  \draw[fill] (\i,5-\i) rectangle (\i+1,5-\i+1);
		  }
	  \foreach \i in {0,1}{
		  \draw[fill] (\i,\i+1) rectangle (\i+1,\i+2);
		  \draw[fill] (\i+1,\i) rectangle (\i+2,\i+1);
		  }
	\end{tikzpicture}
	}\,,\quad
\ka(\splus{1}) &= 
\scalebox{0.22}{
	\begin{tikzpicture}[shift={(2.5,2.5)}, rotate=270]
	  \foreach \i in {0,...,5}{
		  \draw (0,\i) -- (5,\i);
		  \draw (\i,0)-- (\i,5);
		  }
	  %
	  \foreach \i in {1,...,4}{
		  \draw[fill] (\i,5-\i) rectangle (\i+1,5-\i+1);
		  }
	  \foreach \i in {0,1}{
		  \draw[fill] (\i,\i+1) rectangle (\i+1,\i+2);
		  \draw[fill] (\i+1,\i) rectangle (\i+2,\i+1);
		  }
	\end{tikzpicture}
	}\,,\quad
&\ka(\splus{2}) &=
\scalebox{0.22}{
	\begin{tikzpicture}[shift={(2.5,2.5)}, rotate=180]
	  \foreach \i in {0,...,5}{
		  \draw (0,\i) -- (5,\i);
		  \draw (\i,0)-- (\i,5);
		  }
	  %
	  \foreach \i in {1,...,4}{
		  \draw[fill] (\i,5-\i) rectangle (\i+1,5-\i+1);
		  }
	  \foreach \i in {0,1}{
		  \draw[fill] (\i,\i+1) rectangle (\i+1,\i+2);
		  \draw[fill] (\i+1,\i) rectangle (\i+2,\i+1);
		  }
	\end{tikzpicture}
	}\,,\quad
\ka(\splus{3}) = 
\scalebox{0.22}{
	\begin{tikzpicture}[shift={(2.5,2.5)}, rotate=90]
	  \foreach \i in {0,...,5}{
		  \draw (0,\i) -- (5,\i);
		  \draw (\i,0)-- (\i,5);
		  }
	  %
	  \foreach \i in {1,...,4}{
		  \draw[fill] (\i,5-\i) rectangle (\i+1,5-\i+1);
		  }
	  \foreach \i in {0,1}{
		  \draw[fill] (\i,\i+1) rectangle (\i+1,\i+2);
		  \draw[fill] (\i+1,\i) rectangle (\i+2,\i+1);
		  }
	\end{tikzpicture}
	}\,,
\\[1.5pt]
\ka(\sminus{0}) = 
\scalebox{0.22}{
	\begin{tikzpicture}
	  \foreach \i in {0,...,5}{
		  \draw (0,\i) -- (5,\i);
		  \draw (\i,0)-- (\i,5);
		  }
	  %
	  \foreach \i in {1,...,4}{
		  \draw[fill] (\i,5-\i) rectangle (\i+1,5-\i+1);
		  }
	  \draw[fill] (0,3) rectangle (1,4);
	  \draw[fill] (3,0) rectangle (4,1);
	\end{tikzpicture}
	}\,,\quad
\ka(\sminus{1}) &= 
\scalebox{0.22}{
	\begin{tikzpicture}[shift={(2.5,2.5)}, rotate=270]
	  \foreach \i in {0,...,5}{
		  \draw (0,\i) -- (5,\i);
		  \draw (\i,0)-- (\i,5);
		  }
	  %
	  \foreach \i in {1,...,4}{
		  \draw[fill] (\i,5-\i) rectangle (\i+1,5-\i+1);
		  }
	  \draw[fill] (0,3) rectangle (1,4);
	  \draw[fill] (3,0) rectangle (4,1);
	\end{tikzpicture}
	}\,,\quad
&\ka(\sminus{2}) &= 
\scalebox{0.22}{
	\begin{tikzpicture}[shift={(2.5,2.5)}, rotate=180]
	  \foreach \i in {0,...,5}{
		  \draw (0,\i) -- (5,\i);
		  \draw (\i,0)-- (\i,5);
		  }
	  %
	  \foreach \i in {1,...,4}{
		  \draw[fill] (\i,5-\i) rectangle (\i+1,5-\i+1);
		  }
	  \draw[fill] (0,3) rectangle (1,4);
	  \draw[fill] (3,0) rectangle (4,1);
	\end{tikzpicture}
	}\,,\quad
\ka(\sminus{3}) = 
\scalebox{0.22}{
	\begin{tikzpicture}[shift={(2.5,2.5)}, rotate=90]
	  \foreach \i in {0,...,5}{
		  \draw (0,\i) -- (5,\i);
		  \draw (\i,0)-- (\i,5);
		  }
	  %
	  \foreach \i in {1,...,4}{
		  \draw[fill] (\i,5-\i) rectangle (\i+1,5-\i+1);
		  }
	  \draw[fill] (0,3) rectangle (1,4);
	  \draw[fill] (3,0) rectangle (4,1);
	\end{tikzpicture}
	}\,.
\end{align*}
\end{gather*}


In practice, when one starts tessellating $\frakf(\tm)$ with images of the basic tiles under $\kappa$, one sees that adjacent coded tiles are laid so that they overlap in one row or column. This leads us to the following: 
\begin{definition}
The coding $\kappa_1:\Sig\to\Mat_4(\ft)$ is the $4$-coding that is the outcome of removing the last row and column from $\kappa$.
\end{definition}

One might think things would have been simpler had we worked directly with the $4$-coding $\ka_1$, without any reference to the $5$-coding. In fact, 
the symmetries of the $G$-action are only available when one considers the $5$-coding. Therefore, it is useful to allow an overlap of length $1$ between the images of $\kappa$ of adjacent basic tiles. The downside of this is that in order to actually use the symmetries of $\kappa$ with respect to the $G$-action it is necessary to prove that the images under $\kappa$ of adjacent basic tiles that appear in $\sigma^\infty(\splus{0})$ indeed agree on a border of length $1$ between them. As this is not required at the moment, the proof is postponed to Section~\ref{sec:applying-the-coding}.

Since the top left entry of 
$$\sigma(\splus{0})
= \mat{\splus{0}& \tilec{0}\\ \tiler{a}&\sminus{0}}$$
is $\splus{0}$, the iterations $\sigma^l(\splus{0})$ converge as $l\to\infty$ to a substitution tiling of $\bbn^2$, which is denoted by $\sigma^\infty(\splus{0})$.
We now restate our main theorem which the reader can finally understand.
\begin{theorem}
Let $\tm$ be the Thue-Morse sequence as in~\eqref{eq: Thue-Morse}.
The substitution $\sigma$ and coding $\kappa_1$ generate $\frakf(\tm)$. More precisely, $$\left\{\frakf_{m,n}(\tm)\right\}_{m=1,n=1}^\infty=\kappa_1\left(\sigma^\infty(\splus{0})\right).$$
\end{theorem}

The coordinatewise equality of the two tilings of $\bbn^2$ by $\ft$ that appear in this theorem was verified with a computer program for large portions of $\bbn^2$. In particular, it was verified for the part of these tilings that appears in Figure \ref{fig:thm_flat}. The main point of this theorem is that it holds for all $(m,n)\in\bbn^2$.

\subsection{Structure of the paper}\label{sec:structure}
The rest of the paper is organised as follows: Tilings like 
\eqref{eq:tm_flat_array} are characterised by certain arithmetic relations called \emph{frame relations}. These relations are recalled in Section \ref{sec:frame-relations}. Section \ref{sec:tiling} contains the proof of Theorem \ref{thm:main2}. Our proof is based on the frame relations and a detailed analysis of the tilings $\sigma^\infty(\splus{0})$ and $\kappa_1\left(\sigma^\infty(\splus{0})\right)$, and this analysis relies on the symmetries of $\sigma$ and $\kappa$ with respect to the $G$-action for square matrices over $\Sig$ as described in \eqref{eq:equivariance} and \eqref{eq:kappa-equivariance}. Our main theorem is applied in Section \ref{sec:uniform-escape-of-mass} in order to get a recursive formula for the escape of mass along sequences of left-shifts of $\tm$ and for the proof Theorem \ref{thm:main}. 

\section{The frame relations}\label{sec:frame-relations}
This section is devoted to arithmetic relations that characterise arrays such as 
\eqref{eq:tm_flat_array}. A version of this characterisation is studied by Lunnon \cite{L}. It plays a central role in \cite{ALN}, and it also plays a central role in the proof of Theorem \ref{thm:main2}.

The following terminology is 
standard 
and convenient for introducing these relations. For simplicity, it is only phrased in the context of subsets of $\bbz^2$ and functions to $\fq$:
\begin{definition}
	\begin{enumerate}
		\item A \emph{shape} is a subset $S\sub\bbz^2$.
		\item A \emph{pattern} is a function $P:S\to\fq$ where $S$ is a shape. In this case $P$ is a \emph{pattern of shape $S$}.
		\item If $S\sub S'$ are shapes and $P$ and $P'$ are patterns of shapes $S$ and $S'$, respectively, then \emph{$P$ is contained in $P'$} if $\left.P'\right|_S=P$. 
	\end{enumerate}
\end{definition}

We will mostly be dealing with rectangular shapes that are aligned with the standard axes of the plane or along axes which are rotated by $45\degree$. In order to describe the latter ones it is convenient to have a $45\degree$ rotation of the plane with integer coefficients. Since we think of the coordinates of the plane as row vectors, it will be defined as a linear action on row vectors. Denote
\begin{equation}\label{eq: R rotation}
	R(x,y): = (x,y)\mat{1& -1\\ 1&1}=(x+y, y-x)\,.
\end{equation}

\begin{definition}
	\begin{enumerate}
		\item An \textit{axes-aligned rectangle} is a shape $S$ for which there exist
		$x_0,y_0\in\bbz$ and $x,y\in\bbn$ for which 
		$$S = \left([1,x]\times[1,y]+(x_0,y_0)\right)\cap\bbz^2\,.$$
		Its \emph{height} and \emph{width} are $x$ and $y$. It is a \emph{square} if $x=y$, and then its \emph{size} is equal to its height and width.
		\item  A \emph{diagonally-aligned rectangle} is a shape $S$ for which there exist
		$x_0,y_0\in\bbz$ and $x,y\in\bbn$ for which 
		$$S = \left(R([1,x]\times[1,y])+(x_0,y_0)\right)\cap\bbz^2\,.$$
		Its \emph{height} and \emph{width} are $x$ and $y$. 
		It is a \emph{diagonally-aligned square} if $x=y$, and then its \emph{size} equals its height and width.
		\item A pattern is \emph{axes-aligned \emph{(respectively, \emph{diagonally-aligned})} rectangular} if its shape is an axes-aligned (respectively, diagonally-aligned) rectangle. In this case, its \emph{height} and \emph{width} are those of its shape. Axes-aligned rectangular patterns are also referred to as matrices. 
	\end{enumerate}
\end{definition}
\subsection{Number walls}\label{sec:number-walls}
\begin{definition}\label{def:frame-relations} A \emph{frame pattern} and the corresponding \emph{frame relation} is one of the following three:
	\begin{enumerate}
		\item A \emph{plus pattern} is a diagonally-aligned square pattern of size $2$ 
		$$\begin{array}{ccc}
			&A&\\
			B&E&C\\
			&D&
		\end{array}.$$
		It satisfies the \emph{plus relation} if 
		$$E^2 = AD + BC\,.$$ 
		\item \label{item:frame-relations-2}
		An \emph{inner frame pattern} is a pattern of the form
		$$\begin{array}{cccccccc}
			&&&A&&&&\\
			&&&0&&&&\\
			&&\iddots&&\ddots&&&\\
			B&0&&&&\ddots&&\\
			&&\ddots&&&&0&C\\
			&&&\ddots&&\iddots&&\\
			&&&&0&&&\\
			&&&&D&&&\\
		\end{array}$$
		where the $0$'s form a diagonally-aligned rectangle of height $k$ and width $w$ (the $\ell_\infty$-distances between the coordinates of $A$ to $B$ and $C$, respectively). It satisfies the \emph{inner frame relation} if 
		$$AD = (-1)^{k(k+w-1)} BC\,,$$
		\item \label{item:frame-relations-3}
		An \emph{outer frame pattern} is a pattern of the form
		$$\begin{array}{cccccccccc}
			&&&&E&&&&&\\
			&&&&A&A'&&&&\\
			&&&&0&&&&&\\
			&&&\iddots&&\ddots&&&&\\
			F&B&0&&&&\ddots&&C'&\\
			&B'&&\ddots&&&&0&C&G\\
			&&&&\ddots&&\iddots&&&\\
			&&&&&0&&&&\\
			&&&&D'&D&&&&\\
			&&&&&H&&&&\\
		\end{array}$$
		with $ABCD\neq0$. The diagonally-aligned rectangular pattern whose corners are $A,B,C,D$ is the corresponding inner frame pattern. It satisfies the \emph{outer frame relation} if
		$$\left(EB'+\left(-1\right)^kFA'\right)CD=\left(HC'+\left(-1\right)^kGD'\right)AB\,.$$
	\end{enumerate}
\end{definition}

\begin{definition}
	A \emph{number wall} is a pattern which satisfies the frame relations of all patterns which it contains.
\end{definition} 

Note that the plus relations force any two diagonally adjacent zeros to propagate into a $2\times2$ square pattern of zeros. This implies that the zero patterns in number walls are arranged in rectangles. 
The inner frame relations ensure that they are squares -- because any diagonally-aligned rectangular pattern of $0$'s $P$ propagates by the plus relation to the smallest square pattern containing $P$. 

\ignore{
The zero patterns in number walls are called \emph{windows}. For an outer frame pattern as in Definition \ref{def:frame-relations} \eqref{item:frame-relations-2}, the \emph{size} of the window is $d+k-1$ where $d,k$ are the side lengths of the diagonally rectangular pattern
in the pattern involved in the outer frame relation. This is exactly the size of the side length of the corresponding 
maximal square window to which the diagonally rectangular pattern of $0$'s propagates to. Note that the maximality follows from the requirement $ABCD\ne 0$.
}

Number walls are tightly connected to Toeplitz determinants. This connection is the reason they are key to our discussion.
For any $\theta = \{b_n\}_{n=1}^\infty\in \fq^\bbn$
denote 
\begin{equation*}
	\frakt\left(\theta\right)\defeq\left\{\frakt_{m,n}\left(\theta\right)\right\}_{m>-\infty, n\geq|m+1|},
\end{equation*}
where for any $n>m\geq0$,
\begin{equation}\label{eq:toeplitz}
	\frakt_{m,n}\left(\theta\right) \defeq \det
	\begin{pmatrix}
		b_{n}&b_{n+1}&\cdots &b_{n+m}\\
		b_{n-1}&\ddots&\ddots&\vdots\\
		\vdots&\ddots&\ddots&b_{n+1}\\
		b_{n-m}&\cdots&b_{n-1}&b_n
	\end{pmatrix},
\end{equation}
and $\frakt_{-1,n}=1$ for all $n\geq0$ and $\frakt_{m,n}=0$ for all $m<-1$ and $n\geq-m-1$. 
The following theorem follows from \cite[page 11]{L} (see also \cite[Corollary 3.6]{ALN}):
\newcommand{\frakg}{\mathfrak{G}}
\begin{theorem}\label{thm:Lunnon}
	For every $\theta\in\fq^\bbn$ the array $\frakt\left(\theta\right)$ is a number wall. Conversely, if $\frakg=\left\{\frakg_{m,n}\right\}_{m>-\infty,n\geq|m+1|}$ is a number wall which satisfies $\frakg_{-1,n}=1$ for all $n\geq0$ and $\frakg_{m,n}=0$ for all $m<-1$ and $n>-m-1$, then the sequence $\theta=\left\{\frakg_{0,n}\right\}_{n=1}^\infty$ satisfies $\frakt\left(\theta\right) = \frakg$.
\end{theorem}

The frame relations in Definition \ref{def:frame-relations} connect the values on the corners of diagonally-aligned rectangular patterns. Therefore, if one is looking for a simple structure in a number wall, it is natural to look for it in directions parallel to axes of slope $\pm 1$. 
This point was already noticed by Lunnon \cite{L09}. An example of a number wall over $\bbf_3$ with such a structure is suggested in \cite[Appendix C]{L09}. Recently, this approach was employed by Garrett and Robertson \cite{GR} where Lunnon's work is extended to fields of characteristic $5$, $7$ and $11$.

A slightly different approach is taken in this paper. Instead of looking for a structure along lines of slopes $\pm1$, the number wall itself is rotated by $45\degree$ clockwise.
The original tiles fit points of $\bbz^2$ for which all the $\ell_1$-distances are even, while the rest of the points are filled with zeros. Apart from making the structure axes-aligned, the resulting tiling is natural from a dynamical point of view. In short, after the rotation, the minimal $\ell_1$-distance to a nonzero entry is precisely the height in the homogeneous space \eqref{eq:homogeneous}. This point of view 
gives another reason to the zero values that are introduced in-between the tiles of the original number wall. This part of the story is not essential for the proofs of Theorems \ref{thm:main} and \ref{thm:main2}. Hence, we choose not to elaborate further on it.

\subsection{Diagonally-aligned number walls}
We now define a family of tilings which we call \emph{diagonally-aligned number walls} and are obtained from number walls by a simple procedure of rotation, inflation and translation, which we now describe formally. If needed, and without a risk of confusion, number walls as in Section~\ref{sec:number-walls} will be referred to as axes-aligned number walls, and the term number wall will refer to both axes-aligned and diagonally-aligned ones according to the context. Recall the definition of the linear map $R$ in \eqref{eq: R rotation}.

\begin{definition}\label{def: rotated shape}
The \emph{inflated rotation of a shape $S$} is the shape
$$
S^R \defeq R(S)\sqcup \left\{v\in\bbz^2\sep  \av{R(S)\cap\{v\pm (1,0),v\pm (0,1)\}}\geq2\right\}.
$$
For any $v_0\in\bbz^2$, the \emph{inflated rotation of a pattern $P$ translated by $v_0$} is the pattern of shape $S^R+v_0$ which is defined by
$$P^{R,v_0}(v) = \begin{cases}
P\left(R^{-1}(v-v_0)\right)& \mbox{ if }v\in R(S)+v_0\\
0 &\mbox{ if }v\in\left(S^R\setminus R(S)\right)+v_0
\end{cases}.$$
\end{definition}
Since for any shape $S$ the image $R(S)$ is contained in $\set{(m,n)\in \bZ^2 \sep m \equiv n(2)}$, and since the inverse map of $R$ is
\begin{equation*}
	R^{-1}(x,y) = (x,y)\mat{\frac12& \frac12\\[2pt] -\frac12&\frac12}=\left(\frac{x-y}{2}, \frac{x+y}{2}\right),
\end{equation*}
if $P$ is a pattern over $S$ then every $(m,n)\in S^R$ satisfies
\begin{equation}\label{eq:explicit}
P^R(m,n) = \begin{cases}
P\left(\frac{m-n}{2}, \frac{m+n}{2}\right)& \mbox{ if }m \equiv  n(2)\\
0&\mbox{ if }m \not\equiv  n(2)
\end{cases}.
\end{equation}

\begin{definition}
A pattern $P$ is a \emph{diagonally-aligned number wall} if there exist a number wall $P'$ and a vector $v_0\in\bbz^2$ such that $P=(P')^{R,v_0}$.
\end{definition} 


For any $\theta \in\fq^\bbn$, define 
$$
\frakf(\theta) \defeq \frakt(\theta)^{R,(1,0)}\,.
$$ 
By \eqref{eq:explicit}, the explicit formula is 
\begin{equation}\label{eq:explicit2}	\frakf_{m,n}\left(\theta\right)=\begin{cases}
\frakt_{\frac{m-n-1}{2},\frac{m+n-1}{2}}\left(\theta\right)&m\not \equiv n(2) \\
	0&m\equiv n(2)
\end{cases}
\end{equation}
for every $m\geq0$ and $n\geq1$.
\ignore{
and
\begin{equation*}
	\frakf\left(\theta\right)\defeq\left\{\frakf_{m,n}\left(\theta\right)\right\}_{m=0,n=0}^\infty\,.
\end{equation*}
}
It should be clear from the definition of the Hankel and Toeplitz determinants in \eqref{eq:tm_hankel} and \eqref{eq:toeplitz} (and their extension to nonpositive and negative rows, respectively) that the relation between the Hankel and Toeplitz tilings of $\theta$ is 
$$
\frakt_{m,n}(\theta) = (-1)^{\tfrac{m(m+1)}{2}}\frakh_{m+1,n-m}(\theta)\,.$$ 
Hence, equation \eqref{eq:explicit2} is consistent with the definition \eqref{eq: frakf tile} of $\frakf(\theta)$, which was already given in the introduction in terms of Hankel determinants.

The following analogue of Theorem \ref{thm:Lunnon} for diagonally-aligned number walls 
is a direct corollary of it.
\begin{theorem}\label{thm:Lunnon-for-flats}
	For every $\theta\in\fq^\bbn$ the array $\frakf\left(\theta\right)$ is 
	a diagonally-aligned number wall. Conversely, if $\left\{\frakg_{m,n}\right\}_{m\geq0, n\geq1}$ is a diagonally-aligned number wall which satisfies $\frakg_{n-1,n}=1$ for all $n\geq1$ and $\frakg_{m,n}=0$ if for all $m\geq0$ and $n>m+1$, then the sequence $\theta=\left\{\frakg_{n+1,n}\right\}_{n=1}^\infty$ satisfies $\frakf\left(\theta\right) = \left\{\frakg_{m,n}\right\}_{m\geq0, n\geq1}$.
\end{theorem}

When one takes Definition~\ref{def:frame-relations} and applies the linear map $R$ as above, one obtains the following:

\begin{definition}\label{def:flat-frame-relations} 
A \emph{diagonally-aligned frame pattern} and the corresponding \emph{frame relation} is one of the following four:
	\begin{enumerate}
		\item \label{item: FR inflation}
		A \emph{parity pattern} is a rectangular pattern of height and width $1$ and $2$ or $2$ and $1$, i.e.,
		\begin{align*}
			&\begin{array}{cc} 
				B&A
			\end{array},
			&&\textrm{ or }& 
			&\begin{array}{c} 
				A\\
				C
			\end{array}.
		\end{align*}
		It satisfies the \emph{parity relation} if $AB=0$ or $AC=0$, respectively.
		\item\label{item: FR ex} An \emph{ex pattern} is a pattern of the form 
		$$\begin{array}{ccc}
			B&0&A\\
			0&E&0\\
			D&0&C
		\end{array}.$$
		It satisfies the \emph{ex relation} if 
		$$E^2 = AD + BC\,.$$ 
		\item \label{item:flat-frame-relations-2} 
		An \emph{inner frame pattern}  is a rectangular pattern of the form
		$$\begin{array}{ccccc}
			B&0&\cdots&0&A\\
			0&\cdots&&\cdots&0\\
			\vdots&&&&\vdots\\
			\vdots&&&&\vdots\\
			0&\cdots&&\cdots&0\\
			D&0&\cdots&0&C\\
		\end{array}.$$
		It satisfies the \emph{inner frame relation} if 
		$$AD = (-1)^{\frac{k-1}{2} \frac{k+z-4}{2}} BC\,,$$
		where $z$ and $k$ are the height and width of the rectangle (the distances between the coordinates of $A$ to $C$ and $B$, respectively, plus $1$). Note that if $ABCD=0$ then the sign term 
		can be ignored, while if $ABCD\neq0$ and additionally the parity relations as in Definition~\ref{def:flat-frame-relations}~\eqref{item: FR inflation} are satisfied, then $z$ and $k$ must be even. 
		\item\label{item: FR outer} An \emph{outer frame pattern} is a rectangular pattern of the form 
		$$\begin{array}{ccccccccc}
			F&0&*&0&\cdots&0&*&0&E\\
			0&B&0&\cdots&&\cdots&0&A&0\\
			B'&0&\cdots&&&&\cdots&0&A'\\
			0&\cdots&&&&&&\cdots&0\\
			\vdots&&&&&&&&\vdots\\
			\vdots&&&&&&&&\vdots\\
			0&\cdots&&&&&&\cdots&0\\
			D'&0&&&&&&0&C'\\
			0&D&0&\cdots&&\cdots&0&C&0\\
			H&0&*&0&\cdots&0&*&0&G
		\end{array}$$
		with $ABCD\neq0$. The rectangular pattern whose corners are $A,B,C,D$ is the corresponding 
		inner frame pattern. It satisfies the \emph{outer frame relation} if
		$$\left(EB'+\left(-1\right)^{\frac{k-1}{2}}FA'\right)CD=\left(HC'+\left(-1\right)^{\frac{k-1}{2}}GD'\right)AB\,.$$
		where $k$ is the corresponding width of the inner frame.
	\end{enumerate}
\end{definition}
Definition~\ref{def:flat-frame-relations} 
allows us to directly check if a pattern is a diagonally-aligned number wall, by checking that all the frame patterns it contains satisfy the 
frame relations.
\ignore{
\begin{definition}
	A \emph{diagonally-aligned number wall} is a pattern which satisfies the diagonally-aligned frame relations of all patterns which it contains.
\end{definition} 
}
A few remarks are in order. 
\begin{itemize}
	\item The parity relations, which do not appear in the definition of axes-aligned number walls, simply guarantees that our pattern (at least as long as its shape is connected) agrees with a pattern of the form $P^{R,v_0}$ for some pattern $P$ and translation vector $v_0\in\bbz^2$.
	\item The ex relations guarantee that in diagonally-aligned number walls the zeros appear in diagonally-aligned rectangles.
	\item The inner frame relations make sure that the patterns of zeros are diagonally-aligned squares, 
	or an infinite increasing union of such.
	\item Over $\ft$ the frame relations have a particularly simple form:
	\begin{itemize}
		\item The ex relation is 
		$$E = AD + BC\,.$$ 
		\item The inner frame relation is \begin{equation}\label{eq:inner-frame-f2}
			AD=BC=0\,,\quad\textrm{ or }\quad A=B=C=D=1\,.
		\end{equation}
		\item The outer frame relation is
		\begin{equation}\label{eq:outer-frame-f2}
			E+F=G+H\,.
		\end{equation}
	\end{itemize}
	
\end{itemize}
\ignore{
As for number walls, the zero patterns in height walls are called \emph{windows}, and for an outer frame pattern as in Definition \ref{def:flat-frame-relations} \eqref{item:flat-frame-relations-2}, the \emph{size} of the window is 
 $d+k-1$ where $d,k$ are as in Definition \ref{def:flat-frame-relations} \eqref{item:flat-frame-relations-2}. \usnote{I adapted the definition of size to match the new $d,k$}For convenience, ex patterns will be referred to as frame relations of size $0$. This is in accordance with a similar terminology for number walls suggested in \cite[Section 5]{L}.
 }



\section{The tiling and its structure}\label{sec:tiling}
In this section the tilings $\sigma^\infty(\splus{0})$ and $\kappa_1\left(\sigma^\infty(\splus{0})\right)$ are analysed. This analysis is based on the symmetries in the definition of $\sigma$ and $\kappa$ and the extension of the $G$-action to large square patterns over $\Sigma$. The outcome of the analysis is a proof of Theorem \ref{thm:main2}.
\subsection{The structure of $\sig^\infty(\splus{0})$}
The following lemma and corollary explore further the $G$-action on large patterns.
\begin{lemma}
Let $l\ge 1$ and let
$\smallmat{A&B\\C&D}\in \Mat_{2^{l+1}}(\Sig)$. Then,
\begin{align}
\label{eq: rho on blocks}
\rho \mat{A&B\\C&D}& = \mat{\iota \rho C&\iota \rho A\\ \iota \rho D&\iota \rho B},\\
%
\label{eq: eta on blocks}
\eta \mat{A&B\\C&D} &= \mat{\iota \eta D&\iota \eta B\\ \iota \eta C&\iota\eta A},\\
%
\label{eq: iota on blocks}
\iota \mat{A&B\\C&D}& = \mat{\iota A&\iota B\\ \iota C&\iota D}.
\end{align}
\end{lemma}
\begin{proof}
Start with \eqref{eq: iota on blocks}. 
Recall Definition~\ref{def: action on matrices}, that $\tgeo_\iota$ acts trivially and that $\psi^k(\iota)=\iota$ for any $k\in\bbz$. Therefore,
\begin{align*}
\iota \mat{A&B\\C&D} &= \tinn_\iota  \mat{A&B\\ C&D} = 
\mat{\tinn_\iota A&\tinn_\iota B\\ \tinn_\iota C&\tinn_\iota D} 
=\mat{\iota A&\iota B\\ \iota C&\iota D} 
\end{align*}
For \eqref{eq: rho on blocks}, note that on the geometric level
$$\tgeo_\rho\mat{A&B\\ C&D} = \mat{\tgeo_\rho C &\tgeo_\rho A \\ \tgeo_\rho D&\tgeo_\rho B}.$$
Also, since any $l\ge 1$ satisfies 
$$
\psi^{l+1}(\rho) = \iota\psi^l(\rho)\,,
$$
this implies that
\begin{align*}
\rho  \mat{A&B\\ C&D}
&= 
\tinn_{\psi^{l+1}(\rho)} \tgeo_\rho\mat{A& B\\ C&D} 
\\[1.5pt]
&=\tinn_\iota \tinn_{\psi^{l}(\rho)} \mat{\tgeo_\rho C & \tgeo_\rho A\\ \tgeo_\rho D&\tgeo_\rho B} 
\\[3pt]
&=\iota \mat{\tinn_{\psi^l(\rho)}\tgeo_\rho C &\tinn_{\psi^l(\rho)} \tgeo_\rho A\\ \tinn_{\psi^l(\rho)}\tgeo_\rho D&\tinn_{\psi^l(\rho)}\tgeo_\rho B}
\\[1.5pt] 
&=  \iota \mat{\rho C & \rho A 
\\ \rho D& \rho B}
\\
&=  \mat{\iota \rho C  &\iota\rho A\\ 
\iota\rho D&\iota\rho B}.
\end{align*}
A similar reasoning shows \eqref{eq: eta on blocks}.
\end{proof}
The following statement is the propagation of the equivariance in Proposition~\ref{prop: equivariance} from 
basic tiles to larger patterns:
\begin{corollary}\label{cor:key}
	For every $s\in\Sigma$, $g\in G$ and $l\geq0$, the $l$\textsuperscript{th} iteration of $\sig$ satisfies 
	\begin{equation*}
		\sigma^l(g s) = g\sigma^l\left(s\right).
	\end{equation*}
\end{corollary}
\begin{proof}
By induction on $l$. The base of the induction is Proposition \ref{prop: equivariance}. In the inductive step it is assumed that for some $l\geq0$, every $g\in G$ and $M\in \Mat_{2^l}(\Sig)$ satisfy
\begin{equation}\label{eq: equivariance2}
\sig(g M) = g\sig(M)\,.
\end{equation}
Due to the axioms of a group action, in order to complete the inductive step it is enough to check \eqref{eq: equivariance2} for $g \in \{\iota, \rho, \eta\}$ and every $M\in \Mat_{2^{l+1}}(\Sig)$. This identity will then propagate to any $g\in G$. Fix
$M = \smallmat{A&B\\C&D}\in \Mat_{2^{l+1}}(\Sig)$, where $A,B,C,D\in \Mat_{2^l}(\Sig)$. 
For $g=\iota$, using \eqref{eq: iota on blocks} and the inductive hypothesis gives
$$\iota \sig\mat{A&B\\ C&D} 
= 
\mat{\iota \sig( A)&\iota \sig(B)\\ \iota \sig( C)&\iota\sig( D)} 
=
\mat{\sig(\iota A)&\sig(\iota B)\\ \sig(\iota C)&\sig(\iota D)}
=
\sig\iota\mat{A&B\\ C&D}.
$$ 
For $g=\rho$, using \eqref{eq: rho on blocks} and the inductive hypothesis gives
$$\rho \sig\mat{A&B\\ C&D} 
= 
\mat{\iota \rho \sig( C)&\iota\rho \sig(A)\\ \iota\rho \sig( D)&\iota \rho\sig( B)} 
=
\mat{\sig(\iota \rho C)&\sig(\iota \rho C)\\ \sig(\iota \rho D)&\sig(\iota\rho B)}
=
\sig\rho\mat{A&B\\ C&D}.$$
For $g=\eta$ one uses \eqref{eq: eta on blocks} and the inductive hypothesis in a similar fashion.
\end{proof}

\begin{remark}
Corollary \ref{cor:key} with $s=\splus{0}$ and $g=\iota$ is a two dimensional manifestation of a standard observation about the Thue-Morse sequence, which is that $\mu^l(\tileb)$ equals $\mu^l(\tilea)$ in reversed order whenever $l$ is odd, where $\mu$ is defined in \eqref{eq:tm-substitution}.
\end{remark}
\begin{example}\label{example: first iterations}
This is a good place to work out a few iterations of $\sig$ on the tile $\splus{0}$ and get a feeling for the 
strength of the structural properties of the emerging substitution tiling stemming out of the group action.
Here are the first three iterations of $\sig$ on $\splus{0}$:
$$
{\small 
\splus{0}\mapsto 
\begin{array}{cc}
\splus{0} & \tilec{0}\\
\tiler{a} & \sminus{0}
\end{array}
\mapsto
\begin{array}{cccc}
\splus{0} & \tilec{0} & \tilo& \tilo\\
\tiler{a} & \sminus{0} & \tilec{0} & \tilo\\
\sminus{1}& \splus{2} & \sminus{0} & \tilec{0} \\
\splus{0}&\sminus{3} &\tiler{b} & \splus{0}
\end{array}
\mapsto
\begin{array}{cccccccc}
\splus{0} &\tilec{0} &\tilo     &\tilo     &\tilo     &\tilo     & \tilo   &\tilo\\
\tiler{a} &\sminus{0}&\tilec{0} &\tilo     &\tilo     &\tilo     & \tilo   &\tilo\\
\sminus{1}&\splus{2} &\sminus{0}&\tilec{0} &\tilo     &\tilo     &\tilo    &\tilo\\
\splus{0} &\sminus{3}&\tiler{b} &\splus{0} &\tilec{0} &\tilo     &\tilo    &\tilo\\
\tiler{a} &\splus{1}&\sminus{2} &\tiler{a} &\sminus{0}&\tilec{0} &\tilo    &\tilo\\
\sminus{1}&\tilec{1} &\tilec{2} &\splus{2}&\tiler{b} &\splus{0} &\tilec{0}&\tilo\\
\splus{0} &\tilec{0} &\tilec{3} &\sminus{3}&\splus{1} &\sminus{2}&\splus{0}&\tilec{0}\\
\tiler{a} &\sminus{0}&\splus{3} &\tiler{a} &\sminus{0}&\splus{3} &\tiler{a}&\sminus{0}
\end{array}}
$$
We wish to reflect about this procedure. The naive way of looking at each step is that we apply $\sig$ to the 
tiles of the arrays. But, we may also notice that for any $l\ge 0$, 
if we denote $A = \sig^l(\splus{0})$, then 
\begin{align}
\nonumber
\sig^{l+2}(\splus{0}) 
&= 
\sig^{l+1}\mat{\splus{0}&\tilec{0}\\ \tiler{a} &\iota \splus{0}} 
\\
\nonumber
&= 
\mat{
\sig(A)&\sig^{l+1}(\tilec{0})\rule{0pt}{2.5ex} \\
{\sig^l\mat{\iota\rho \splus{0}&\rho^2\splus{0} \\ \splus{0}&\iota\rho^3\splus{0}}}
&\iota \sig(A)}
\\[1.5pt]
\label{eq:magic}
&=
\mat{
\sig(A) & \sig^{l+1}(\tilec{0})\\
{\begin{array}{cc}
\iota \rho A&\rho^2A\\
A&\iota \rho^3A
\end{array}
}
& \iota\sig(A)
}.
\end{align}
Note what happens in the bottom left quarter of \eqref{eq:magic}: it is built from four copies of $A$ which differ from one another by the action of $\iota\rho$. This is nicely demonstrated in the above iteration for $l=2$. It is now very easy to apply $\sig$ again and calculate
$\sig^4(\splus{0})$ from $\sig^2(\splus{0})$ and $\sig^3(\splus{0})$: 
$$
\begin{array}{cccccccccccccccc}
	\splus{0} &\tilec{0} &\tilo     &\tilo     &\tilo     &\tilo     & \tilo   &\tilo & \tilo &\tilo &\tilo     &\tilo     &\tilo     &\tilo     & \tilo   &\tilo\\
	\tiler{a} &\sminus{0}&\tilec{0} &\tilo     &\tilo     &\tilo     & \tilo   &\tilo& \tilo &\tilo &\tilo     &\tilo     &\tilo     &\tilo     & \tilo   &\tilo\\
	\sminus{1}&\splus{2} &\sminus{0}&\tilec{0} &\tilo     &\tilo     &\tilo    &\tilo& \tilo &\tilo &\tilo     &\tilo     &\tilo     &\tilo     & \tilo   &\tilo\\
	\splus{0} &\sminus{3}&\tiler{b} &\splus{0} &\tilec{0} &\tilo     &\tilo    &\tilo& \tilo &\tilo &\tilo     &\tilo     &\tilo     &\tilo     & \tilo   &\tilo\\
	\tiler{a} &\splus{1}&\sminus{2} &\tiler{a} &\sminus{0}&\tilec{0} &\tilo    &\tilo& \tilo &\tilo &\tilo     &\tilo     &\tilo     &\tilo     & \tilo   &\tilo\\
	\sminus{1}&\tilec{1} &\tilec{2} &\splus{2}&\tiler{b} &\splus{0} &\tilec{0}&\tilo& \tilo &\tilo &\tilo     &\tilo     &\tilo     &\tilo     & \tilo   &\tilo\\
	\splus{0} &\tilec{0} &\tilec{3} &\sminus{3}&\splus{1} &\sminus{2}&\splus{0}&\tilec{0}& \tilo &\tilo &\tilo     &\tilo     &\tilo     &\tilo     & \tilo   &\tilo\\
	\tiler{a} &\sminus{0}&\splus{3} &\tiler{a} &\sminus{0}&\splus{3} &\tiler{a}&\sminus{0}& \tilec{0} &\tilo &\tilo     &\tilo     &\tilo     &\tilo     & \tilo   &\tilo
	\\
	\sminus{1}&\splus{2}&\tiler{b}&\sminus{1}&\splus{2}&\tiler{b}&\sminus{1}&\splus{2}&\sminus{0} &\tilec{0} &\tilo     &\tilo     &\tilo     &\tilo     & \tilo   &\tilo
	\\
	\splus{0}&\sminus{3}&\splus{1}&\tilec{1}&\tilec{2}&\sminus{2}&\splus{0}&\sminus{3}&\tiler{b} &\splus{0}&\tilec{0} &\tilo     &\tilo     &\tilo     & \tilo   &\tilo
	\\
	\tiler{a}&\splus{1}&\tilec{1} &\tilo&\tilo&\tilec{2}&\sminus{2}&\tiler{a}&\splus{1}&\sminus{2} &\splus{0}&\tilec{0} &\tilo     &\tilo     &\tilo    &\tilo
	\\
	\sminus{1}&\tilec{1}&\tilo&\tilo&\tilo&\tilo&\tilec{2}&\splus{2}&\sminus{0} &\splus{3}&\tiler{a} &\sminus{0} &\tilec{0} &\tilo     &\tilo    &\tilo
	\\
	\splus{0}&\tilec{0}&\tilo&\tilo&\tilo&\tilo&\tilec{3}&\sminus{3}&\tiler{b} &\sminus{1}&\splus{2} &\tiler{b} &\splus{0}&\tilec{0} &\tilo    &\tilo
	\\
	\tiler{a}&\sminus{0}&\tilec{0} &\tilo&\tilo&\tilec{3}&\splus{3}&\tiler{a}&\splus{1}&\tilec{1} &\tilec{2} &\sminus{2}&\tiler{a} &\sminus{0} &\tilec{0}&\tilo
	\\
	\sminus{1}&\splus{2}&\sminus{0}&\tilec{0}&\tilec{3}&\splus{3}&\sminus{1}&\splus{2}&\sminus{0} &\tilec{0} &\tilec{3} &\splus{3}&\sminus{1} &\splus{2}&\sminus{0}&\tilec{0}
	\\
	\splus{0}&\sminus{3}&\tiler{b}&\splus{0}&\sminus{3}&\tiler{b}&\splus{0}&\sminus{3}&\tiler{b} &\splus{0}&\sminus{3} &\tiler{b} &\splus{0}&\sminus{3}&\tiler{b}&\splus{0}
\end{array}.
$$
Note the perfect diamond shape of the 
tile $\tilo$ that appears in the bottom left corner, contoured by a fence of tiles $\tilec{i}$ where $i$ changes from one edge of the fence to the other due to the action of $\rho$. This phenomenon repeats itself in further iterations of $\sigma$. See Figure~\ref{fig:thm_infls} for $\sigma^5(\splus{0})$.
\end{example}
\begin{figure}[ht]
	\centering
	\includegraphics[width=1\linewidth]{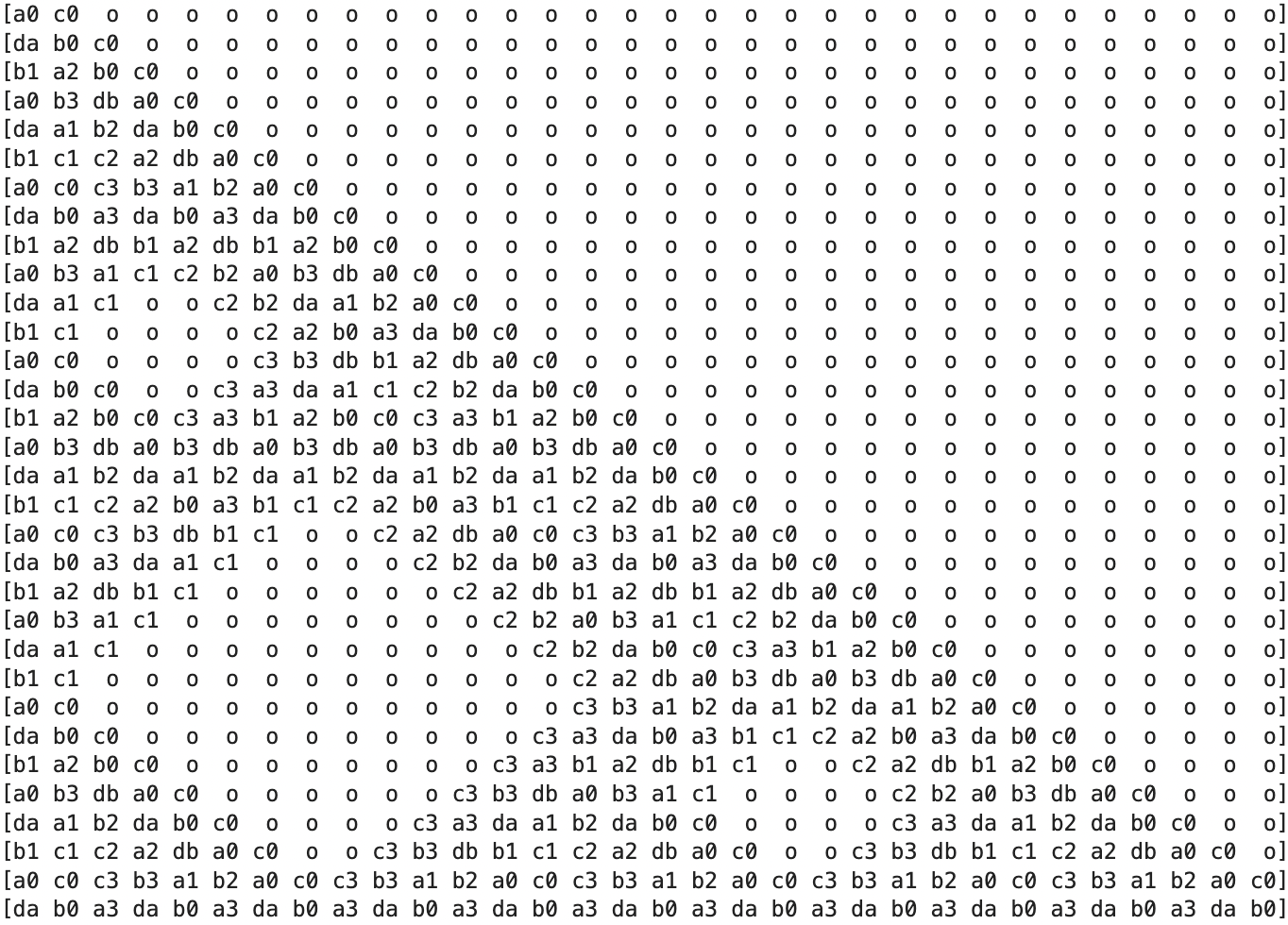}
	\caption{SageMath output  of the pattern $\sigma^{5}(\splus{0})$, where the tile symbols have been altered in a suggestive manner.}
	\label{fig:thm_infls}
\end{figure}
\ignore{
By Definition~\ref{def: iota} and Proposition~\ref{prop: equivariance} (see also \eqref{eq: acting with eta on 2 matrix})
$$\sigma(\sminus{0}) = \sigma(\iota \splus{0}) = \iota\sigma(\splus{0}) = \iota\left(\begin{array}{cc}
	\splus{0} &\tilec{0} \\
	\tiler{a} &\sminus{0}
\end{array}\right) = \begin{array}{cc}
\sminus{0} & \tilec{0} \\
\tiler{b} & \splus{0}
\end{array}$$
}
The following four lemmas reveal bits of information regarding the structure of the tiling $\sig^\infty(\splus{0})$. We will use them in order to obtain relevant information about its coding $\ka_1\left(\sig^\infty(\splus{0})\right)$ in the next subsection.

\begin{lemma}\label{lem:sigma0_diagonal}
For any $l\geq0$, the main diagonal of $\sigma^l(\splus{0})$ is the Thue-Morse sequence on the symbols
$\splus{0}$ and $\sminus{0}$ of length $2^l$. The diagonal which lies just above it is a sequence of $\tilec{0}$'s, and apart from that every entry above the main diagonal is $\tilo$. A similar statement
holds for $\sig^\infty(\splus{0})$. Analogous statements hold for $\sigma^l(\sminus{0})$ and $\sig^\infty(\sminus{0})$.
\end{lemma}
\begin{proof}
The claims about the main diagonals follow by inspecting the main diagonals of 
\begin{align*}
	\sigma(\splus{0}) &= \mat{\splus{0}& \tilec{0} \\ \tiler{a}&\sminus{0}}, &&\textrm{ and } &\sigma(\sminus{0}) &= \mat{\sminus{0} &\tilec{0} \\ \tiler{b} & \splus{0}}.
\end{align*} 
The claims about the diagonals above them follow by also inspecting 
\begin{align*}
	\sigma(\tilec{0}) &= \mat{\tilo &\tilo\\ \tilec{0} &\tilo}, &&\textrm{ and } &\sigma(\tilo) &= \mat{\tilo&\tilo\\ \tilo&\tilo}.
\end{align*}
The reader should consult Example~\ref{example: first iterations}.
\end{proof}

We will use the following terminology: 

\begin{definition}
	Let $\wh{\Sig}$ is a finite set of tiles and $P:\{1,\ldots,m\}\to\wh{\Sig}$ be a finite sequence of tiles from $\wh{\Sig}$.
	\begin{enumerate}
		\item The sequence $P^k$ is the concatenation of $P$ with itself $k$ times, i.e., every $1\leq i \leq m$ and $0\leq k'<k$ satisfies
		$$
		P^k(i+k'm) = P(i)\,.
		$$ 
		\item A finite sequence $P'$ of tiles from $\wh{\Sig}$ is \emph{periodic with period $P$} if there exists $k\ge 1$ such that $P'$ starts with $P^k$ and $P^{k+1}$ starts with $P'$.
	\end{enumerate}
\end{definition}

The following lemma allows us to identify when a
$1$-dimensional substitution tiling is periodic with a given period. Its proof is left to the reader.

\begin{lemma}\label{lem: periodic}
Let $\wh{\Sig}$ be a finite set and $\wh{\sig}:\wh{\Sig}\to\wh{\Sig}^k$ be a $1$-dimensional $k$-substitution on $\wh{\Sig}$, where $k\ge2$.  
If $s\in \wh{\Sig}$ and $l_0\ge 0$ satisfy that 
both $\wh{\sig}^{l_0}(s)$ and 
$\wh{\sig}^{l_0+1}(s)$ are periodic with period $P$, then 
$\wh{\sig}^l(s)$ is periodic with period $P$ for any $l\ge l_0$.
\end{lemma}

\begin{lemma}\label{lem:sigmal0}
	Let $l\geq0$. 
	\begin{enumerate}
		\item The first column of $\sigma^l(\splus{0})$ is periodic with period
	\begin{align*}
		\begin{array}{c} \splus{0}\\ \tiler{a} \\ \sminus{1}\end{array}.
	\end{align*}
	\item The last row of $\sigma^l(\splus{0})$ is periodic with period 
	\begin{align*}
		\begin{array}{ccc} 
			\splus{0} & \sminus{3} & \tiler{b}    
		\end{array},\quad &\textrm{for $l$ even}\,,
		\\
		\begin{array}{ccc} 
			\tiler{a} & \sminus{0} & \splus{3}
		\end{array},\quad &\textrm{for $l$ odd}\,.
	\end{align*}
	\item The first column of $\sigma^l(\sminus{0})$ is periodic with period
	\begin{align*}
		\begin{array}{c} \sminus{0} \\ \tiler{b} \\ \splus{1} \end{array}.
	\end{align*}
	\item The last row of $\sigma^l(\sminus{0})$ is periodic with period
	\begin{align*}
		\begin{array}{ccc} 
			\sminus{0} & \splus{3} & \tiler{a}
		\end{array},\quad &\textrm{for $l$ even}\,,
		\\
		\begin{array}{ccc}
			\tiler{b} & \splus{0} & \sminus{3}
		\end{array},\quad &\textrm{for $l$ odd}\,.
	\end{align*}
\end{enumerate}
\end{lemma}


\begin{proof}
The reader should consult Example~\ref{example: first iterations}.

\ignore{
	The following few iterations of $\sig$ on $\splus{0}$ can assist in the proof:
\begin{equation*}
{\small
\splus{0}\mapsto 
\begin{array}{cc}
\splus{0} & \tilec{0}\\
\tiler{a} & \sminus{0}
\end{array}
\mapsto
\begin{array}{cccc}
\splus{0} & \tilec{0} & \tilo& \tilo\\
\tiler{a} & \sminus{0} & \tilec{0} & \tilo\\
\sminus{1}& \splus{2} & \sminus{0} & \tilec{0} \\
\splus{0}&\sminus{3} &\tiler{b} & \splus{0}
\end{array}
 \mapsto
\begin{array}{llllllll}
\splus{0} & \tilec{0}& \tilo& \tilo& \tilo& \tilo& \tilo& \tilo\\
\tiler{a} & \sminus{0}& \tilec{0} & \tilo& \tilo& \tilo& \tilo& \tilo\\
\sminus{1}& \splus{2}& \sminus{0} & \tilec{0}& \tilo& \tilo&\tilo& \tilo\\
\splus{0}&\sminus{3}&\tiler{b} & \splus{0}&\tilec{0} & \tilo&\tilo& \tilo\\
\tiler{a}&\splus{1}&\sminus{2}&\tiler{a}&\sminus{0} & \tilec{0}&\tilo& \tilo\\
\sminus{1}&\tilec{1}&\tilec{2}&\splus{2}&\tiler{b}&\splus{0}& \tilec{0} & \tilo\\
\splus{0} & \tilec{0}& \tilec{3}& \sminus{3}&\splus{1}& \sminus{2}&\splus{0} & \tilec{0}\\
\tiler{a} & \sminus{0}&\splus{3}&\tiler{a} &\sminus{0}&\splus{3}&\tiler{a} & \sminus{0}
\end{array}
}
\end{equation*}
}

\begin{enumerate}
	\item The first column of $\sig^\infty(\splus{0})$ is a $1$-dimensional $2$-substitution tiling of $\bN$ obtained by restricting the substitution $\sig$ to the first column. Let us write it explicitly for the reader (written horizontally for convenience): 
	$$\splus{0}\mapsto \splus{0}\; \tiler{a}\,;\quad \tiler{a}\mapsto \sminus{1}\; \splus{0}\,;\quad \sminus{1}\mapsto \splus{0}\; \tiler{a}\,.$$
	The claim follows by applying Lemma~\ref{lem: periodic} and noticing that already the second and third iterations are periodic
	with same period:
	$$\splus{0}\mapsto \splus{0}\; \tiler{a}\mapsto \splus{0}\; \tiler{a}\; \sminus{1}\; \splus{0}\mapsto 
	\splus{0}\; \tiler{a}\; \sminus{1}\; \splus{0}\; \tiler{a}\; \sminus{1}\; \splus{0}\; \tiler{a}\,.
	$$
\ignore{
As can be seen, there are only three symbols in this substitution tiling. Moreover, it readily follows by induction that in this substitution, at each finite stage an appearance of $\splus{0}$ is followed by $\tiler{a}$, an appearance of $\tiler{a}$ is followed by $\sminus{1}$ and an appearance of $\sminus{1}$ is followed by $\splus{0}$. This finishes the proof of the first claim in the statement. 
}

	\item The second claim follows from equivariance: recall that $\eta\splus{0} = \splus{0}$ and 
	so that $\sig^l(\splus{0}) = \sig^l(\eta\splus{0})=
	\eta\sig^l(\splus{0})$. The action of $\eta$ on matrices of size $2^l$ is obtained by reflecting along the anti-diagonal and then either applying  $\eta$ or $\iota \eta$ to the entries, according to whether $l$ is even or odd, respectively. Reflecting along the anti-diagonal takes the first column to the last row in reversed order. By the first claim, the first column of $\sig^l(\splus{0})$ is periodic with period
	\begin{align*}
		\begin{array}{c} \splus{0}\\ \tiler{a} \\ \sminus{1}\end{array}.
	\end{align*}
	Taking this period in a reversed order and applying $\psi^l(\eta)$ gives
	\begin{align*}
		\mat{
			 \eta\sminus{1} & \eta\tiler{a} & \eta\splus{0}
		}&=\mat{ 
		\sminus{3} & \tiler{b} & \splus{0}
		}, &&\textrm{for $l$ even}\,,
		\\
		\mat{
			\iota\eta\sminus{1} & \iota\eta\tiler{a} & \iota\eta\splus{0}
		}&=\mat{
		\splus{3} & \tiler{a} & \sminus{0}
		}, &&\textrm{for $l$ odd}\,.
	\end{align*}
	Since $2^l\equiv(-1)^l(3)$, the last column of $\sig^l(\splus{0})$ starts from the last tile in this period if $l$ is even, and from the second if $l$ is odd. This gives the second claim.
	
	\item The third claim also follows from the equivariance of the $G$-action with respect to $\sig$ in Corollary~\ref{cor:key},
	since $\sig^l (\sminus{0})=\sig^l( \iota \splus{0}) = \iota \sig^l(\splus{0})$ and the action of $\iota$ on matrices amounts to applying $\iota$ to the entries. In particular this applies to the first column. 

	\item The forth claim follows along the same lines, for example, by using the equivariance $\sig^l(\iota \splus{0}) = \iota \sig^l(\sminus{0})$ and the second claim as an input. 
\end{enumerate}
\end{proof}

\begin{lemma}\label{lem:sigmal8}
	Let $l\geq-1$. 
	\begin{enumerate}
		\item The first row of $\sigma^{l+1}\left(\tiler{a}\right)$ is periodic with period
		\begin{equation}\label{eq:sigmal8_row}
			\begin{aligned}
				\begin{array}{ccc} \sminus{1}&\splus{2}&\tiler{b} 	\end{array},\quad &\textrm{for $l$ even}\,,
				\\
				\begin{array}{ccc} 	\tiler{a}&\splus{1}&\sminus{2}
				\end{array},\quad &\textrm{for $l$ odd}\,.
			\end{aligned}
		\end{equation}
		\item The last column of $\sigma^{l+1}\left(\tiler{a}\right)$ is periodic with period
		\begin{equation}\label{eq:sigmal8_column}
			\begin{aligned}
				\begin{array}{c}  \splus{2} \\ \sminus{3}\\ \tiler{a}  \end{array},\quad &\textrm{for $l$ even}\,,&
				\begin{array}{c} \tiler{a} \\ \splus{2} \\ \sminus{3} \end{array},\quad &\textrm{for $l$ odd}\,.
			\end{aligned}
		\end{equation}
		\item The first row of $\sigma^{l+1}\left(\tiler{b}\right)$ is periodic with period 
		\begin{equation*}
			\begin{aligned}	
				\begin{array}{ccc} \splus{1}&\sminus{2}&\tiler{a}
				\end{array},\quad 
				&\textrm{for $l$ even}\,,
				\\
				\begin{array}{ccc} \tiler{b}&\sminus{1} &\splus{2}
				\end{array},\quad 
				&\textrm{for $l$ odd}\,.
			\end{aligned}
		\end{equation*}
		\item The last column of $\sigma^{l+1}\left(\tiler{b}\right)$ is periodic with period 
		\begin{equation*}
			\begin{aligned}	
				\begin{array}{c} 
					\sminus{2}\\ 
					\splus{3} \\ 
					\tiler{b}
				\end{array},\quad 
				&\textrm{for $l$ even}\,,&
				\begin{array}{c} 
					\tiler{b}\\ 
					\splus{2}\\ 
					\splus{3} 
				\end{array},\quad
				&\textrm{for $l$ odd}\,.
			\end{aligned}
		\end{equation*}
	\end{enumerate}
\end{lemma}

\begin{proof}
	The statements are obvious for $l=-1$. For every $l\geq0$, the $(l+1)$\textsuperscript{st} iteration of $\sigma$ on $\tiler{a}$ can be expressed using the $l$\textsuperscript{th} iteration of $\sigma$ as  
	\begin{equation}\label{eq:sigmal8}
		\sigma^{l+1}\left(\tiler{a}\right) = \mat{
		\sigma^{l}(\sminus{1}) &\sigma^{l}(\splus{2}) \\
		\sigma^{l}(\splus{0}) &\sigma^{l}(\sminus{3})
		} = 
		\mat{
		\sigma^{l}(\rho \sminus{0}) &\sigma^{l}(\rho^2\splus{0}) \\
		\sigma^{l}(\splus{0}) &\sigma^{l}(\rho^3\sminus{0})
		}.
	\end{equation}
	\begin{enumerate}
		\item By \eqref{eq:sigmal8}, the first row of $\sig^{l+1}\left(\tiler{a}\right)$ is obtained by concatenating the first rows of $\sig^l(\rho\sminus{0})=\rho\sig^l(\sminus{0})$ and $\sig^l(\rho^2\splus{0}) = \rho^2\sig^l(\splus{0})$. 
		Recall that the action of $\rho$ on a matrix amounts to rotating it $90\degree$ clockwise and acting with either $\rho$ or $\iota\rho$ on each of the entries, according to whether the size of the matrix is an even or odd power of $2$, respectively. Applying Lemma~\ref{lem:sigmal0}, which 
		includes the information about the first column of $\sig^l(\sminus{0})$ and the last row of $\sig^l(\splus{0})$, which transform under $\rho$ and $\rho^2$ into these rows in reversed order:
		\begin{itemize}
		\item If $l$ is even then the first row of $\sig^{l+1}\left(\tiler{a}\right)$ is obtained by concatenating two periodic sequences of the same period 
		$$\splus{2}\; \tiler{b}\;\sminus{1}$$ 
		\item If $l$ is odd then the first row of $\sig^{l+1}\left(\tiler{a}\right)$ is obtained by concatenating
		two periodic sequences of the same period 
		$$\sminus{2}\; \tiler{a}\;\splus{1}$$
		\end{itemize}
		For the conclusion of the claim 
		it is enough to justify why the concatenation is done in a way which gives \eqref{eq:sigmal8_row}. By Lemma~\ref{lem: periodic} (applied to $\wh{\sigma}=\sigma^2$) this consistency can be checked by verifying it on the first few iterations of this substitution tiling, which can be done by examining Figure~\ref{fig:thm_infls}.
	
		\item Recall that $\rho \tiler{a} = \tiler{a}$. It follows from the equivariance that $\sig^{l+1}\left(\tiler{a}\right) = \rho\sig^{l+1}\left(\tiler{a}\right)$, so by the first claim, the last column of $\sig^{l+1}\left(\tiler{a}\right)$ is obtained from its first row by applying to its entries $\iota\rho$ or $\rho$ according to whether $l$ is even or odd, respectively, and writing 
		it vertically. 
		By \eqref{eq:sigmal8_row}, this gives 
		\begin{equation*}
			\begin{aligned}
				&&
				\begin{array}{c}
					\iota\rho\sminus{1}\\ \iota\rho\splus{2} \\ 
					\iota\rho\tiler{b}
				\end{array} = 
				\begin{array}{c} 
					\splus{2} \\ 
					\sminus{3}\\ 
					\tiler{a}
				\end{array},\quad 
				&\textrm{for $l$ even}\,,&&
				\begin{array}{c}
					\rho\tiler{a} \\
					\rho\splus{1} \\ 
					\rho\sminus{2}
				\end{array} = 
				\begin{array}{c} 
					\tiler{a} \\
					\splus{2} \\ 
					\sminus{3}
				\end{array},\quad 
				&\textrm{for $l$ odd}\,,
			\end{aligned}
		\end{equation*}
		which is precisely \eqref{eq:sigmal8_column}.
		\ignore{
		By \eqref{eq:sigmal8}, the last column of 
		$\sig^{l+1}\left(\tiler{a}\right)$ is obtained by stacking the last column of $\sig^l(\rho^2\splus{0})$ on top of the last 
		column of $\sig^l(\sminus{3})$. Using equivariance gives that the last column of 
		$\sig^l(\rho^2\splus{0})=\rho^2\sig(\splus{0})$ is the first column of $\sig^l(\splus{0})$ written in reversed order after acting on the tiles with $\rho^2$. By  Lemma~\ref{lem:sigmal0} this amounts to saying that this column is periodic with period
		$$\begin{array}{ccc}
		\sminus{3}\\
		\tiler{a}\\
		\splus{2}
		\end{array}.$$
		As explained, this column is stacked above the last column of $\sig^l(\sminus{3}) = \iota\rho^3\sig^l(\splus{0})$. 
		By similar considerations, this equals
		$$\begin{array}{ccc}
		\sminus{3}\\
		\tiler{a}\\
		\splus{2}
		\end{array}.
		$$
		Stacking these displayed periods one above the other establishes the third claim.
		}
		
		\item Since $\sig^l\left(\tiler{b}\right) = \iota\sig^l\left(\tiler{a}\right)$, the third claim follows from \eqref{eq:sigmal8_row}.
		
		\item By the same reasoning, the fourth claim follows from \eqref{eq:sigmal8_column}.
	\end{enumerate}
\ignore{		
	
	Recall that $\sigma^3=\rho1$ and $4=\rho^20$. Therefore, by Corollary \ref{cor:key}, the first row of $\sigma^l(3)$ and $\sigma^l(4)$ are determined by the first column of $\sigma^l(1)$ and the last row of $\sigma^l(0)$, respectively. The last column of $\sigma^l(4)$ is determined by the first column of $\sigma^l(0)$, while since $7=\rho^3(1)$, the last column of $\sigma^l(7)$ is determined by the last row of $\sigma^l(1)$. Therefore, Lemma \ref{lem:sigmal0} implies \eqref{eq:sigmal8_row} and \eqref{eq:sigmal8_column}.
	Similarly, the fact that $\eta8=9$ gives  
	\begin{equation}\label{eq:sigma9}
	\sigma(9) = \begin{array}{cc}
		2 &5 \\
		1 &6
	\end{array},
	\end{equation}
	so the definition of $\sigma$ gives
	$$\sigma^{l+1}(9) = \begin{array}{cc}
		\sigma^{l}(2) &\sigma^{l}(5) \\
		\sigma^{l}(1) &\sigma^{l}(6)
	\end{array}.$$
	Therefore, since $2=\rho0$ and $5=\rho^21$, Lemma \ref{lem:sigmal0} implies \eqref{eq:sigmal9_row} and \eqref{eq:sigmal9_column} in a similar manner to the above.
}
\end{proof}

\subsection{Applying the coding}\label{sec:applying-the-coding}

Prior to the consideration of $\kappa_1\left(\sigma^\infty(\splus{0})\right)$, the following observation is useful:

\begin{lemma}\label{lem:consistent}
	The coding $\kappa$ and the tiling $\sigma^\infty(\splus{0})$ are consistent with respect to an overlap of length $1$, in the sense of \cite[Definition 4.8]{ALN}.
\end{lemma}

\begin{proof}
	The claim means that every square pattern of size $2$	
	$$\begin{array}{cc}
		a & b \\
		c & d
	\end{array}$$
	which is contained in $\sigma^\infty(\splus{0})$, satisfies 
	\begin{equation}\label{eq:consistent1}
		\left.\kappa(a)\right|_{[1,5]\times\{5\}} = \left.\kappa(b)\right|_{[1,5]\times\{1\}}\quad \text{and} \quad 
		\left.\kappa(c)\right|_{[1,5]\times\{5\}} = \left.\kappa(d)\right|_{[1,5]\times\{1\}}\,,
	\end{equation}
	and
	\begin{equation}\label{eq:consistent2}
	\left.\kappa(a)\right|_{\{5\}\times[1,5]} = \left.\kappa(c)\right|_{\{1\}\times[1,5]}\quad\text{and}\quad
	\left.\kappa(b)\right|_{\{5\}\times[1,5]} = \left.\kappa(d)\right|_{\{1\}\times[1,5]}\,.
	\end{equation} 
	It is a finite check verified with a computer program that there are $64$ square patterns of size $2$ in both $\sigma^5(\splus{0})$ and $\sigma^6(\splus{0})$ (see Figure \ref{fig:thm_infls}). Therefore, these are precisely the square patterns of size $2$ in $\sigma^\infty(\splus{0})$ (cf. \cite[Lemma 4.14]{ALN}). Verifying \eqref{eq:consistent1} and \eqref{eq:consistent2} for them is a finite check which is done with a computer program.
\end{proof}

\begin{convention}\label{conv: overlap}
Due to Lemma \ref{lem:consistent}, and following \cite[Section 4.3]{ALN}, it makes sense to extend the definition of the image under $\kappa_1$ of any rectangular pattern $P$ in $\sigma^\infty(\splus{0})$ by one column and one row. This means that if $P$ is a rectangular pattern of height and width $m\times n$ of tiles in $\Sig$, then $\kappa_1(P)$ is a rectangular pattern of height and width $(4m+1)\times(4n+1)$ of tiles in $\ft$ which is composed out of the images under $\ka$ of the tiles in $P$ stacked together with an overlap of length $1$. 
\end{convention}
\begin{convention}\label{conv: overlap2}
	In order to handle overlaps as in Convention \ref{conv: overlap}, it is helpful to introduce the following notation: given $k,l,m,n\geq1$ and $mn$ matrices of height and width $k+1$ and $l+1$, such that
	\begin{align*}
		A_{i,j}(k,\cdot) &= A_{i+1,j}(1,\cdot)\,,
		&&\textrm{for every $1\leq i\leq m-1$ and $1\leq j\leq 	n\,,$} &&\textrm{and}\\
		A_{i,j}(\cdot,l) &= A_{i,j+1}(\cdot,1)\,,&&
		\textrm{for every $1\leq i\leq m$ and $1\leq j\leq n-1$}\,,&&
	\end{align*}
 	then the matrix with square brackets
	\begin{equation}\label{eq:convention}
	\br{\begin{array}{ccc}
			A_{1,1}&\cdots&A_{1,n}\\
			\vdots &\vdots & \vdots\\ 
			A_{m,1}&\cdots&A_{m,n}
		\end{array}
	}
	\end{equation}
	is a matrix of height $km+1$ and width $ln+1$ which is composed out of the matrices $A_{i,j}$ arranged as blocks in a combined matrix in a way so that any two adjacent matrices overlap in $1$ row or column according to their respective position. Formally, for every $1\leq i \leq km+1$ and $1\leq j \leq ln+1$, denote 
	\begin{align*}
		i'&= 	\min\left\{\left\lceil\frac{i}{k}\right\rceil,m\right\},&&\textrm{ and } &
		j'&=
		\min\left\{\left\lceil\frac{j}{l}\right\rceil,n\right\}.
	\end{align*}
	Then the $(i,j)$\textsuperscript{th} entry of the matrix in \eqref{eq:convention} is
	$$
	A_{i',j'}
	\left(
	i - k\left(i'-1\right),
	j - l\left(j'-1\right)
	\right).
	$$
\end{convention}

These two conventions are particularly useful for the statement and proof of the following lemma regarding the symmetries of $\kappa_1$ with respect to the action of $\rho$.
 
\begin{lemma}\label{lem: kappa equivariance}
For every $l\ge 0$ and $s\in \Sig$ the image of $\rho \sig^l(s)$ under $\kappa_1$ satisfies
\begin{equation}\label{eq: kappa equiv}
\ka_1\left(\rho \sig^l(s)\right) = \rho \ka_1\left(\sig^l\left(\iota^l s\right)\right)\,.
\end{equation}
\end{lemma}
\begin{proof}
The proof is by induction on $l$. For $l=0$, equation \eqref{eq: kappa equiv} reads as 
$\ka_1(\rho(s)) = \rho\ka_1(\rho s)$. This equality holds by the way $\ka_1$ was defined, which equals $\ka$ by Convention~\ref{conv: overlap} -- on representatives of $\rho$-orbits and extended by the action of $\rho$. 

Assuming that the equality \eqref{eq: kappa equiv}  holds for all $s\in \Sig$ for some $l\ge 0$, it will be proved that it holds for all $s\in \Sig$ for $l+1$. Given any tile $s\in \Sig$, there are tiles $s_1,s_2,s_3,s_4\in\Sig$ such that
$$
\sig(s) = \mat{s_1 &s_2\\ s_3&s_4}.
$$ 
Therefore,
$$\sig^{l+1}(s) = \mat{\sig^l(s_1) &\sig^l(s_2)\\ \sig^l(s_3)&\sig^l(s_4)}.$$
By \eqref{eq: rho on blocks} and Corollary~\ref{cor:key}, this gives that
$$\rho\sig^{l+1}(s) 
= 
\mat{
\rho\iota\sig^l(s_3) &\rho\iota \sig^l(s_1)\\ \rho\iota \sig^l(s_4)&\rho\iota\sig^l(s_2)
}
=
\mat{
\rho \sig^l(\iota s_3)&\rho \sig^l(\iota s_1)\\ \rho \sig^l(\iota s_4)&\rho\sig^l(\iota s_2)
}
.$$
Therefore,
\begin{equation}\label{eq: non-trivial}
\ka_1\left(\rho \sig^{l+1}(s)\right) =  \br{\begin{array}{cc}
\ka_1\left(\rho \sig^l(\iota s_3)\right)&\ka_1\left(\rho \sig^l(\iota s_1)\right)\\[2pt]
\ka_1\left(\rho \sig^l(\iota s_4)\right)&\ka_1\left(\rho \sig^l(\iota s_2)\right)
\end{array}
}.
\end{equation} 
The right hand side of \eqref{eq: non-trivial} needs to be interpreted correctly: 
each block $\ka_1\left(\rho\sig^l(\iota s_i)\right)$ is of size $2^{l+2}+1$ by convention, and they are arranged in the combined matrix such that any two adjacent matrices overlap in $1$ row or column according to their respective position in the block matrix. This way they make together a matrix of size $2^{l+3}+1$ as they should. 

The fact that the four blocks within these square brackets can be glued like that in a consistent way requires a justification. Indeed, Corollary~\ref{cor:key} implies that $\rho$ and $\sig^l$ commute. Therefore, each of the $4$ blocks of \eqref{eq: non-trivial} appear in $\sig^\infty(\splus{0})$, and thus this consistency property is satisfied by Lemma~\ref{lem:consistent}. 

The equality in \eqref{eq: non-trivial} is then immediate: the application of $\ka_1$ is by Convention~\ref{conv: overlap} obtained by applying $\ka$ to all the entries of a matrix and gluing the resulting $5\times 5$ matrices over $\bF_2$
with an overlap of length $1$. The right hand side of \eqref{eq: non-trivial} is the exact same procedure taking into account the square bracket notation in Convention~\ref{conv: overlap2}. 

To finish, the inductive hypothesis and Corollary~\ref{cor:key} are applied to \eqref{eq: non-trivial} in order to conclude that
\begin{align*}
\ka_1\left(\rho \sig^{l+1}(s)\right)& = 
  \br{
\begin{array}{cc}
\rho\ka_1\left(\sig^l\left(\iota^{l+1}s_3\right)\right)&\rho\ka_1\left(\sig^l\left(\iota^{l+1}s_1\right)\right)\\
\rho\ka_1\left( \sig^l\left(\iota^{l+1}s_4\right)\right)&\rho\ka_1\left(\sig^l\left(\iota^{l+1}s_2\right)\right)
\end{array}
}\\[3pt]
&=
\rho \br{
\begin{array}{cc}
\ka_1( \sig^l(\iota^{l+1}s_1))&\ka_1(( \sig^l(\iota^{l+1}s_2))\\
\ka_1( \sig^l((\iota^{l+1}s_3) )&\ka_1((\sig^l(\iota^{l+1}s_4))
\end{array}
}
\\[1.5pt]
&=
\rho \ka_1\left(\sig^{l+1}\left(\iota^{l+1}s\right)\right),
\end{align*}
where the last equality follows from \eqref{eq: iota on blocks}.

\end{proof}
\ignore{
\begin{remark}
Note that there is no analogue of Lemma~\ref{lem: kappa equivariance} without 
Convention~\ref{conv: overlap}. That is, even though it seems cleaner to define $\ka_1$ to be the 
$4$-coding introduced in the introduction, the only way we can see and exploit the symmetries is with adding 
this convention.
\end{remark}
}

Before we state the next lemma, we recall that $\mu$ and $\tau$ stand for the substitution  \eqref{eq:tm-substitution} which generates the Thue-Morse sequence and the coding \eqref{eq:tm-coding} which sends $\tilea$ to $0$ and $\tileb$ to $1$. 
\begin{lemma}\label{lem:kappa1sigma0_diagonal}
	For any $l\ge0$, the main diagonal of $\kappa_1\left(\sigma^l(\splus{0})\right)$ is a sequence of $0$s. The diagonal below it is 
	the sequence $\tau\left(\mu^{l+2}(\tilea)\right)$,  
	the diagonal above it is a sequence of $1$'s, and apart from that every entry above the main diagonal is $0$. 
	An analogous statement holds for $\sminus{0}$, where $\tilea$ is replaced with $\tileb$.
\end{lemma}
\begin{proof}
	The first claim follows from Lemma \ref{lem:sigma0_diagonal} by recalling \eqref{eq:kappa} and noting that $\kappa_1(\splus{0})$ and $\kappa_1(\sminus{0})$ have only $0$'s on the main diagonal. 
	
	The second claim follows from Lemma \ref{lem:consistent}, the fact that $\kappa(\splus{0})$ 
	and $\kappa(\sminus{0})$ have $\tau(\mu^2(\tilea))$ and $\tau(\mu^2(\tileb))$, respectively, on the diagonal which lies just below the main diagonal, and from the fact that both $s\in\{\tilea,\tileb\}$ satisfy $\mu^2\left(\mu^\infty(s)\right)=\mu^\infty(s)$. 
	
	The third claim follows by inspecting the diagonals above the main diagonal of $\ka_1(\splus{0})$ and $\ka_1(\sminus{0})$ which contain only $1$'s.
	
	The fourth claim follows by inspecting $\kappa(\tilec{0})$ and $\kappa(\tilo)$. 
	The assertion about $\sminus{0}$ is proved in a similar fashion.
\end{proof}
\begin{lemma}\label{lem: kappa1 of r0}
For any $l\ge0$ the pattern $\ka_1(\sig^l\left(\tiler{a}\right))$ has the form
\begin{center}
	\scalebox{0.45}{
		\begin{tikzpicture}

			\draw (-9.25,-9.25) rectangle (9.25,9.25);
			\draw (-8.75,0.25)--(-0.25, 8.75);
			\draw (-8.75,-0.25)--(-0.25, -8.75);
			
			\draw (0.25,-8.75)--(8.75,-0.25);
			\draw (0.25,8.75)--( 8.75,0.25);
			\draw (-9.25,-0.25) rectangle (-8.75, 0.25);
			\draw (9.25,-0.25) rectangle (8.75, 0.25);
			\draw (-0.25,-9.25) rectangle ( 0.25,-8.75);
			\draw (0.25,9.25) rectangle ( -0.25,8.75);
			
			%
			
			\node at (-9,0) {\LARGE $0$};
			\node at (9,0) {\LARGE $0$};
			\node at (0,9) {\LARGE $0$};
			\node at (0,-9){\LARGE $0$};
			
			\foreach \i in {0,1,...,5}{
				\node at(-8.5 +\i/2,\i/2) {\LARGE $1$};
				\node at (-8.5 +\i/2 +1/2, \i/2) {\LARGE $0$};
			}
			\foreach \i in {0,1,...,5}{
				\node at(-8.5 +\i/2,-\i/2) {\LARGE $1$};
				\node at(-8.5 +\i/2+1/2,-\i/2) {\LARGE $0$};
			}
			
			\foreach \i in {0,1,...,5}{
				\node at(\i/2,-8.5 +\i/2) {\LARGE $1$};
				\node at(\i/2,-8.5 +\i/2 +1/2) {\LARGE $0$};
			}
			\foreach \i in {0,1,...,5}{
				\node at(-\i/2,-8.5 +\i/2) {\LARGE $1$};
				\node at(-\i/2,-8.5 +\i/2+1/2) {\LARGE $0$};
			}
			
			\foreach \i in {0,1,...,5}{
				\node at(8.5 -\i/2,\i/2) {\LARGE $1$};
				\node at(8.5 -\i/2-1/2,\i/2) {\LARGE $0$};
			}
			\foreach \i in {0,1,...,5}{
				\node at(8.5 -\i/2,-\i/2) {\LARGE $1$};
				\node at(8.5 -\i/2-1/2,-\i/2) {\LARGE $0$};
			}
			
			
			\foreach \i in {0,1,...,5}{
				\node at(\i/2,8.5 -\i/2) {\LARGE $1$};
				\node at(\i/2,8.5 -\i/2-1/2) {\LARGE $0$};
			}
			\foreach \i in {0,1,...,5}{
				\node at(-\i/2,8.5 -\i/2) {\LARGE $1$};
				\node at(-\i/2,8.5 -\i/2-1/2) {\LARGE $0$};
			}
			
			\node (0,0){\huge only $0$'s};
			\node at (-6,-6) {\Huge *};
			\node at (6,-6) {\Huge *};
			\node at (-6,6) {\Huge *};
			\node at (6,6) {\Huge *};
			
			\draw [decorate,decoration={brace,amplitude=8pt},xshift=0pt,yshift=0pt]
			(-10,-9.2) -- (-10,9.2) node [midway,left,xshift=0pt] {\huge \begin{tabular}{l} 
				length  \\ $2^{l+2}+1\;$ 
			\end{tabular}};
	\end{tikzpicture}}
\end{center}

\end{lemma}
\begin{proof}
Writing  
$A = \sig^{l-1}(\splus{0})$, 
equation~\eqref{eq:magic} implies that  
$$\sig^l\left(\tiler{a}\right) = \mat{
\iota \rho A &\rho^2A\\ A &\iota \rho^3 A
}.$$
Applying $\ka_1$ and using the bracket notation of Lemma~\ref{lem: kappa equivariance} gives
$$\ka_1\left(\sig^l\left(\tiler{a}\right)\right) = 
\br{
\begin{array}{cc}
\ka_1(\iota \rho A) & \ka_1(\rho^2A)\\
\ka_1(A)&\ka_1(\iota \rho^3 A)
\end{array}
}.$$ 
The pattern that appears in the statement of the theorem is verified quarter by quarter. The statement for the bottom left quarter $\ka_1(A)$ is included in Lemma~\ref{lem:kappa1sigma0_diagonal}. 
Next, it is verified that the upper left quarter $\ka_1(\iota \rho A)$ has the desired pattern.
Lemma~\ref{lem: kappa equivariance} implies that
$$\ka_1(\iota \rho A) =\ka_1\left(\rho\iota \sig^{l-1}(\splus{0})\right) = \rho\ka_1\left(\sig^{l-1}\left(\iota^l\splus{0}\right)\right).$$
The claim follows from Lemma~\ref{lem:kappa1sigma0_diagonal} applied to $\splus{0}$ or $\sminus{0}$ according to the parity of $l$, and the definition of $\tgeo_\rho$ which rotates $90\degree$ clockwise. The argument for the other two quarters is virtually the same.
\end{proof}	
	
\subsection{Proof of Theorem~\ref{thm:main2}}	
Lemma \ref{lem:kappa1sigma0_diagonal} and the converse part of Theorem \ref{thm:Lunnon-for-flats} reduce Theorem \ref{thm:main2} to the verification of the frame relations in the tiling $\ka_1\left(\sig^\infty(\splus{0})\right)$. 

\begin{theorem}\label{thm:main2-reduced}
Let $\set{\frakg_{m,n}}_{m,n\ge 1} := \ka_1\left(\sig^\infty(\splus{0})\right)$ and extend it to $m=0$ by setting $\frakg_{0,1}=1$ and $\frakg_{0,n}=0$ for every $n>1$. The tiling $\frakg\defeq\left\{\frakg_{m,n}\right\}_{m\geq0, n\geq1}$ is a diagonally-aligned number wall.
\end{theorem}

\begin{proof}

The proof is by verification of the frame relations of Definition~\ref{def:flat-frame-relations}. It is useful to note that for every finite subset of the set $\{(m,n)\sep m\geq0,n\geq1\}$ the frame relations can be checked in a simple manner -- since the first part of Theorem \ref{thm:Lunnon-for-flats} implies that $\frakf(\alpha)$ satisfies the frame relations, it is enough to verify that $\frakg$ agrees with $\frakf(\alpha)$ on that set. As mentioned in Section \ref{sec:structure}, this was verified in particular for the subset $\{(m,n)\sep 0\leq m\leq129,1\leq n\leq130\}$ (see Figure \ref{fig:thm_flat}). This fact will be used twice during this proof. 

\quad\\
\noindent\emph{Parity relations} \eqref{item: FR inflation}. Examining
the matrices in the definition of $\ka$ 
in \eqref{eq:kappa} and its equivariance with respect to $\rho$ in \eqref{eq:kappa-equivariance}, it is evident that for every $s\in\Sigma$, its coding $\kappa(s)$ satisfies $\ka(s)(m,n)=0$ for every $1\leq m,n\leq 5$ for which $m\equiv n(2)$. Therefore, $\frakg_{m,n}=0$ for every $m,n\geq 1$ for which $m\equiv n(2)$. By the definition of the $0$\textsuperscript{th} row, this property is easily checked to extend to the row $m=0$. Therefore, the tiling $\frakg$ satisfies the parity relations.

\begin{center}
\scalebox{0.8}{
\begin{tikzpicture}
    \draw[thick, ->] (-1, 0) -- (4, 0) node[right] {$n$};
    
    \draw[thick, ->] (0, 1) -- (0, -4) node[below] {$m$};
    
    \foreach \i in {0.5,1,1.5,...,3} \node at (\i, -\i) {0}; 
    
    \foreach \i in {1,1.5,...,3}\node at (\i, 0) {0}; 
    

    \foreach \i in{0.5,1,1.5,...,3} \node at (\i, -\i+0.5) {1}; 
    
    \foreach \i in{1,1.5,...,3} \node at (\i, -\i+1) {0}; 
    
    \node at (0.5,-1) {0};
    \node at (1,-1.5){1};
    \node at (1.5, -2){1};
    \node at (2,-2.5){0};
    \node at (2.5,-3){1};
    \draw[->] (-1,-2)--(1.3,-2);
    \node[left] at (-1,-2) {Thue-Morse};
    \node at (3.25, -0.75) {only $0$'s...};
    \node at (1, -3.25) {\small {\begin{tabular}{c} interesting \\ part \end{tabular}}};
\end{tikzpicture}}
\end{center}

\noindent\emph{Ex relations} \eqref{item: FR ex}. 
The ex patterns whose underlying shapes are not contained in $\bbn^2$ are clearly satisfied. The rest of them are contained in the image under $\ka_1$ of a $2\times 2$ matrix over $\Sig$. As noted in the proof of Lemma~\ref{lem:consistent}, all the square patterns of size $2$ that appear 
in $\sig^\infty(\splus{0})$, must already appear in $\sig^5(\splus{0})$. Checking that $\ka_1(\sig^5(\splus{0}))$ satisfies the ex relation is a finite check that is verified with a computer program. In fact, by the first paragraph of the current proof, this holds since $\frakg_{m,n}=\frakf_{m,n}(\alpha)$ whenever $0\leq m\leq129$ and $1\leq n \leq 130$, and $\frakf(\alpha)$ does satisfy the ex relations. 


\quad\\
\noindent \emph{Inner \eqref{item:flat-frame-relations-2} and outer \eqref{item: FR outer} frame relations}.
\ignore{Let $P$ be an inner frame pattern in $Q'$. If it is not contained in $Q$, then its top corners are 0's and the corresponding frame relation holds. Moreover, such a pattern is not part of an outer frame pattern. 
It follows that in order to verify that $Q'$ satisfies the inner and outer frame relations, it is enough to restrict attention to frame patterns contained in $Q$.
}
Every inner frame pattern whose shape is not contained in $\bbn^2$, has top corners that are $0$'s and, thus, the corresponding frame relation holds. Moreover, such a pattern is not contained in any outer frame pattern. It follows that in order to verify that $\frakg$ satisfies the inner and outer frame relations, it is enough to restrict attention to frame patterns that are contained in $\bbn^2$.

The proof proceeds by induction on $l$ in order to show the following:
\begin{claim}\label{claim: induction}
For any $l\ge0$, the inner and outer frame relations hold in
$\ka_1\left(\sig^l(s)\right)$ for any $s\in \Sig$. 
\end{claim}
The base of the induction is verified by examining $\ka(s)$ for $s\in \Sig$ and is left to the reader. 
In the inductive step, it is assumed that Claim \ref{claim: induction} holds for $l$ and proved that it holds for $l+1$. The analysis splits into cases according to $s\in\Sigma$. 

\quad\\
\noindent \textbf{Case 1}. For $s = \tilo$ or $s\in G\tilec{0}$, then it is immediate that all the frame patterns in $\ka_1\left(\sig^{l+1}(s)\right)$ satisfy the corresponding frame relations, because for any inner frame pattern there, at least three of its corners are $0$'s. This also implies that there are no outer frame patterns in $\ka_1\left(\sig^{l+1}(s)\right)$. 
\quad\\

\ignore{
that for any $s\in \Sig$, $\ka_1(\sig^l(s))$ is thought of as a 
pattern of size $(2^{l}+1)\times(2^{l}+1)$ and that the 1-overlapping property allows us to 
obtain $\ka_1(\sig^{l+1}(s))$ by placing the 4 quarters of it with one row or column overlap. 
}
\ignore{
Denote 
\begin{align*}
	\sig(s) &= \mat{a&b\\ c&d},  
	&&\mbox{ and }
	&\ka_1\left(\sig^{l+1}(s)\right)
	&= 
	{\br{ \begin{array}{cc} \ka_1\left(\sig^l(a)\right)&\ka_1\left(\sig^l(b)\right)\\ \ka_1\left(\sig^l(c)\right)&\ka_1\left(\sig^l(d)\right)\end{array}}}.
\end{align*}
}

\quad\\
\noindent\textbf{Case 2}.	
For $s=\splus{0}$, it is convenient to introduce some terminology. Recall Convention~\ref{conv: overlap} and the square bracket notation that is introduced in Convention~\ref{conv: overlap2}. Consider the equality
$$
\kappa_1\left(\sigma^{l+1}(\splus{0})\right) 
= 
\br{
\begin{array}{cc}
	\ka_1\left(\sigma^{l}(\splus{0})\right) &\ka_1\left(\sigma^{l}(\tilec{0})\right) \\
	\ka_1\left(\sigma^{l}\left(\tiler{a}\right)\right) &\ka_1\left(\sigma^{l}(\sminus{0})\right)
\end{array}}.
$$
There exists a rectangular pattern of height $9$ and width $2^{l+3}+1$ which arises from the overlapping 
images under $\ka_1$ of the last and first rows of the top and bottom halves of $\sig^{l+1}(\splus{0})$, respectively. 
This rectangular pattern will be referred to as the \emph{middle horizontal strip}. A similar terminology is used for the \emph{middle vertical strip}.

Let $P$ be an inner frame pattern contained in $\kappa_1\left(\sigma^{l+1}(\splus{0})\right)$, and let $A,B,C,D$ be its corners, using the notation of Definition~\ref{def:flat-frame-relations}\eqref{item:flat-frame-relations-2}. If $P$ is contained in one of the quarters, the inductive hypothesis applies. Otherwise, 
either $A$ and $B$ are in the upper half and $C$ and $D$ are in the lower half, or $A$ and $C$ are in the right half and $B$ and $D$ are in the left one. Assume, for example, the former. By Lemmas \ref{lem:sigmal0} and \ref{lem:sigmal8}, the middle horizontal strip is periodic with a prescribed period given by the lemmas. If the width of $P$ is strictly greater than $5$, then the middle horizontal strip must contain a rectangular pattern of zeros of width strictly greater than $5$. 
However, it does not. 

Indeed, the middle horizontal strip is the image under $\kappa_1$ of one of two periodic strips which are composed of rectangular patterns of height $2$ and width $3$: either 
\begin{equation}\label{eq:first-four-1}
\begin{array}{cccc}
		\splus{0} &\sminus{3}&\tiler{b}& \cdots\\ \tiler{a}&\splus{1}&\sminus{2}&\cdots
\end{array}
\end{equation}
for $l$ even, or
\begin{equation}\label{eq:first-four-2}
	\begin{array}{cccc}
		\tiler{a}&\sminus{0}&\splus{3}&\cdots\\
		\sminus{1}&\splus{2} &\tiled{b}&\cdots
	\end{array}
\end{equation}
for $l$ odd. Therefore, it is enough to check that there are no more than $5$ consecutive $0$'s in 
$$
\br{
	\begin{array}{cccc}
		\kappa_1(\splus{0}) &\kappa_1(\sminus{3})&\kappa_1\left(\tiler{b}\right)&\kappa_1(\splus{0}) \\ \kappa_1(\tiled{a})&\kappa_1(\splus{1})&\kappa_1(\sminus{2})&\kappa_1(\tiled{a})
\end{array}}
$$
and
$$
\br{
	\begin{array}{cccc}
		\kappa_1\left(\tiler{a}\right)&\kappa_1(\sminus{0})&\kappa_1(\splus{3})&\kappa_1\left(\tiler{a}\right)\\
		\kappa_1(\sminus{1})&\kappa_1(\splus{2}) &\kappa_1\left(\tiler{b}\right)&\kappa_1(\sminus{1})
\end{array}}.
$$
The patterns which make up the first $4$ columns of \eqref{eq:first-four-1} and \eqref{eq:first-four-2} already appear in the first $4$ column in rows $4\textrm{--}5$ and $8\textrm{--}9$, respectively, as can be verified by looking at $\sig^3(\splus{0})$ in Example \ref{example: first iterations}. Therefore, it is enough to look at Figure~\ref{fig:thm_flat} and verify that the images under $\kappa_1$ of these pattern do not contain strictly more than $5$ consecutive $0$'s in a row. These images appear in columns $1\textrm{--}17$ and rows $13\textrm{--}21$ and $29\textrm{--}37$, respectively.
 
The height of $P$ must also be smaller than or equal to $5$. Otherwise, the ex relations would imply that the middle horizontal strip contains a diagonally-aligned square of $0$'s of size strictly greater than $5$. So, in particular, 
a rectangle of $0$'s of width strictly greater than $5$,
which was just ruled out. 

\ignore{The reader can use the following depiction of the horizontal middle strip in order to verify these statements.
\begin{center}
\begin{tikzpicture}
    \foreach \x in {0,1,...,5} \node at (\x, 0.5) {0};
    \foreach \x in {0.5,1.5,...,5.5} \node at (\x, 0.5) {*};
    \foreach \x in {0,1,...,5} \node at (\x, 1) {*};
    \foreach \x in {0.5,1.5,...,5.5} \node at (\x, 1) {0};
    \foreach \x in {0,1,...,5} \node at (\x, 1.5) {0};
    \foreach \x in {0.5,1.5,...,5.5} \node at (\x, 1.5) {*};
    \foreach \x in {0,1,...,5} \node at (\x, 2) {*};
    \foreach \x in {0.5,1.5,...,5.5} \node at (\x, 2) {0};
    %
    \foreach \x in {0,1,...,5} \node at (\x, 0) {*};
    \foreach \x in {0.5,1.5,...,5.5} \node at (\x, 0) {0};
    \foreach \x in {0,1,...,5} \node at (\x, -0.5) {0};
    \foreach \x in {0.5,1.5,...,5.5} \node at (\x, -0.5) {*};
    \foreach \x in {0,1,...,5} \node at (\x,- 1) {*};
    \foreach \x in {0.5,1.5,...,5.5} \node at (\x, -1) {0};
    \foreach \x in {0,1,...,5} \node at (\x, -1.5) {0};
    \foreach \x in {0.5,1.5,...,5.5} \node at (\x, -1.5) {*};
    \foreach \x in {0,1,...,5} \node at (\x,- 2) {*};
    \foreach \x in {0.5,1.5,...,5.5} \node at (\x, -2) {0};

    \draw[thick] (-0.2, -0.2) rectangle (5.75, 0.25);
    
    \node at (7, 0) {\small{overlap row}};
     \node at (7, 1) {\small{middle strip}};

\end{tikzpicture}
\end{center}
}
	
A similar analysis applies to inner frame patterns which involve the middle vertical strip. It is enough to verify that the patterns
$$
\br{
	\begin{array}{cc}
		\kappa_1\left(\tiler{a}\right)&\kappa_1(\sminus{0})\\
		\kappa_1(\splus{2})&\kappa_1\left(\tiler{b}\right)\\
		\kappa_1(\sminus{3})&\kappa_1(\splus{1})\\  \kappa_1\left(\tiler{a}\right)&\kappa_1(\sminus{0})
\end{array}}
$$
and 
$$
\br{
	\begin{array}{cc}
		\kappa_1(\splus{2})&\kappa_1(\sminus{0})\\
		\kappa_1(\sminus{3})&\kappa_1\left(\tiler{b}\right)\\  
		\kappa_1\left(\tiler{a}\right)&\kappa_1(\splus{1})\\
		\kappa_1(\splus{2})&\kappa_1(\sminus{0})
\end{array}}
$$
do not contain strictly more than $5$ consecutive $0$'s in a column. The relevant patterns in $\sig^\infty(\splus{0})$ occur already in rows $5\textrm{--}8$ and columns $4\textrm{--}5$, and in rows $9\textrm{--}12$ and columns $8\textrm{--}9$, respectively, as can be verified by looking at $\sig^4(\splus{0})$ in Example \ref{example: first iterations}. Therefore, it is enough to look at Figure~\ref{fig:thm_flat} and check their images under $\kappa_1$, which occur at rows $17\textrm{--}33$ and columns $13\textrm{--}21$, and rows $33\textrm{--}51$ and columns $29\textrm{--}37$, respectively.

The conclusion is that every inner frame pattern which is not contained in one of the quarters of $\sig^{l+1}(\splus{0})$ must have height and width smaller than or equal to $5$. Since it is assumed that it is not contained in one of the quarters, this implies that it is fully contained in the middle horizontal strip or the middle vertical strip. As mentioned above, these strips are periodic, so verifying that such a pattern must satisfy the inner frame relation boils down to a finite check. Recall that by \eqref{eq:inner-frame-f2}, over $\ft$ the inner frame relations are equivalent to the property that if any three out of the corners $A,B,C,D$ are $1$'s, then the fourth corner must be $1$ as well. Therefore, in order to verify the inner frame relation it is enough to verify that the $0$ tiles inside these middle strips occur in diagonally-aligned squares (which may be rectangular patterns contained in such squares in case they lie near the boundary of the strip). Therefore, this finite check can be carried out by seeing that the $0$'s in $\kappa_1\left(\sigma^4(\splus{0})\right)$ are arranged in diagonally-aligned squares. Alternatively, this follows from the fact that $\kappa_1\left(\sigma^4(\splus{0})\right)$ is contained in $\frakf(\alpha)$ as was discussed in the beginning of the current proof, and hence satisfies all the frame relations.

Assume now that $P'$ is an outer frame pattern extending an inner frame pattern $P$ with $ABCD\ne 0$, referring to the notation of Definition~\ref{def:flat-frame-relations}\eqref{item: FR outer}. If $P'$ is contained in one of the quarters, the inductive hypothesis applies and the corresponding outer frame relation holds. If $P'$ meets two quarters, then a similar analysis to the above shows that in fact it must be contained the corresponding middle strip. Again, since these middle strips are periodic with explicit periods which already occur in $\kappa_1\left(\sig^4(\splus{0})\right)$, and since $\kappa_1\left(\sig^4(\splus{0})\right)$ satisfies the frame relations -- this finishes the inductive step for $s=\splus{0}$.

A similar argument works for $s=\sminus{0}$. Since it is clear from \eqref{eq:inner-frame-f2} and \eqref{eq:outer-frame-f2} that over $\ft$ the frame relations are invariant under $\rho$, by Lemma~\ref{lem: kappa equivariance} the same holds for every $s\in \set{\splus{i}, \sminus{i}:0\le i\le 3}$. 	
\quad\\

\noindent\textbf{Case 3}. For $s=\tiler{a}$, let $P$ be an inner frame pattern in $\ka_1\left(\sig^{l+1}\left(\tiler{a}\right)\right)$ with corners $A,B,C,D$. If $P$ is contained in one of the quarters, the inductive hypothesis applies and the corresponding frame relation holds. If some non-opposite corners of the pattern are $0$'s then the corresponding inner frame relation holds. Applying Lemma~\ref{lem: kappa1 of r0} and taking into account that in an inner frame pattern, all the values are $0$'s apart from potentially the corners, gives that the complimentary cases are exactly when the corners are sitting on the diagonally-aligned square of $1$'s which appears in the statement of that lemma. In this case, the inner frame relation holds by recalling their form over $\ft$ in \eqref{eq:inner-frame-f2}.
	
Assume now that $P'$ is an outer frame pattern in $\ka_1\left(\sig^{l+1}\left(\tiler{a}\right)\right)$ and let $P$ be the inner frame pattern which corresponds to it, with corners $A,B,C,D$ such that $ABCD\ne 0$. If $P'$ is contained in one of the quarters of $\ka_1\left(\sig^{l+1}\left(\tiler{a}\right)\right)$, then the inductive hypothesis applies. If not, the previous paragraph implies that the corners $A,B,C,D$ sit on the diagonally-aligned square of $1$'s in $\ka_1\left(\sig^{l+1}\left(\tiler{a}\right)\right)$. Using the notation in Definition~\ref{def:flat-frame-relations}\eqref{item: FR outer}, observe that if $E=H$ and $F=G$ then the outer frame relation holds, by recalling its simple form over $\ft$ in \eqref{eq:outer-frame-f2}. This is verified now. As shown in the proof of Lemma~\ref{lem: kappa1 of r0}, 
\begin{equation}\label{eq: 4 quarters}
	\ka_1\left(\sig^{l+1}\left(\tiler{a}\right)\right) = \br{
	\begin{array}{c c}
	M_2 &{\rho^2}M_1\\ M_1 & {\rho^2} M_2
	\end{array}}.
\end{equation}
where $M_1=\ka_1\left(\sig^l(\splus{0})\right)$ and $M_2 =\rho \left(\ka_1\left(\sig^l(s)\right)\right)$ for $s\in \set{\splus{0},\sminus{0}}$. In any case, the conclusion is that indeed $H=E$ and $F=G$. This finishes the inductive step for $s=\tiler{a}$. The analysis for $\tiler{b}$ is similar. This finishes the inductive step for all cases and verifies Claim~\ref{claim: induction}, which finishes the proof of the theorem.
\end{proof}

\section{Escape of mass of the Thue-Morse sequence}\label{sec:uniform-escape-of-mass}
In this section we compute the escape of mass of $t^j\tm$ where $\tm$ is the Thue-Morse sequence thought of as a series in $\ftttt$. Escape of mass is expressed in Definition~\ref{def:escape} in terms of degrees of partial quotients of Laurent series. As we explain shortly, these degrees are tightly related to the length of consecutive zeros in number walls. However, in order to use the self-similarity of the number wall $\frakf(\alpha)$, it is convenient to extend the definition of escape of mass to number walls in general.

\subsection{Motivational picture}
We start by explaining an intuitive picture that motivates our discussion and terminology. Imagine a space $X$ partitioned into a sequence of subsets $X = \bigsqcup_{i=0}^\infty X_i$. The function $\mb{h}:X\to\bZ_{\ge 0}$ defined by $\mb{h}(x) =i\iff x\in X_i$ will be thought of as a height function, and so we think of $X$ as a tower with $X_i$ being the $i$\textsuperscript{th} floor. Imagine that there is a map $T:X\to X$ having the following two properties:
\begin{enumerate}
\item For all $x\in X$, $\mb{h}(T(x))\in \set{\mb{h}(x)+1, \mb{h}(x)-1}\cap \bZ_{\ge 0}$. That is, the height must change under the action of $T$ and cannot change by more than one.
\item If $\mb{h}(x) = i>0$, and $\mb{h}(T(x))=i-1$, then for any $0\le j\le i$, $\mb{h}(T^j(x)) = i-j$. That is, once an orbit starts going down, it will go down until it reaches floor $0$.
\end{enumerate}
Fix $x\in X_0$ and suppose we are given a sequence $P=\set{P_i}_{i=1}^\infty\in \bF_q^\bN$ with the property that 
$P_i \ne 0\iff T^i(x)\in X_0$. Then, we can recover the height of $T^i(x)$ by studying the $0$'s in the sequence $P$: each time we see in the sequence $P$ a pattern of the form $P_l,0,\dots, 0,P_{l+r}$, with $P_l\cdot P_{l+r}\ne 0$, then it follows that $r$ is even and the height of $T^{l+i}(x)$  satisfies 
\begin{equation}\label{eq:motivational}
	\mb{h}\left(T^{l+i}(x)\right) = \begin{cases}
	i & \mbox{ for } 0\le i\le r/2\\
	r/2 - (i-r/2) &\mbox{ for } r/2+1\le i \le r.
	\end{cases}
\end{equation}
This is exactly what happens behind the scenes when one restricts attention to a sequence that appears as a column (or a row) of a diagonally-aligned number wall. 

\subsection{Escape of mass for sequences}
In this section we develop the terminology which is necessary for the consideration of escape of mass in the context of general sequences over a finite field.
\begin{definition}\label{def:escape-for-sequences}
	Let $P:[m_1,m_2]\to \fq$ be any sequence over $\fq$ which is not identically zero. The \emph{height of $P$} is a function $\height_P:[m_1,m_2]\to\bbn$ defined by
	$$
	\height_P(i) \defeq \min_{P(i')\neq0} |i'-i|\,.
	$$
	The \emph{escape of mass of $P$} is a function $\esc_{P}:\bbn\to\bbn$ defined by
	$$
	\esc_P(d) \defeq \frac{\#\left\{m_1\leq i<m_2 \sep \max\left\{\height_P(i),\height_P(i+1)\right\}>d\right\}}{m_2-m_1}\,.
	$$
We refer to $\esc_P(d)$ as \emph{the mass of $P$ that lies above height $d$}.
\end{definition}

A basic tool that will be used several times later in this section is the ability to compute the escape of mass of a finite sequence by considering a partition of it into subsequences. In the statement of the next lemma we will use the square bracket notation that was introduced in Convention~\ref{conv: overlap2}.

\begin{lemma}\label{lem:gluing}
	Let $P':[0,m']\to \fq$ and $P'':[0,m']\to \fq$ be sequences over $\fq$ such that $P'(m')=P''(0)$, and set 
	$$P=\br{\begin{array}{c} P'\\P'' \end{array}}.$$
	Assume that $P'$ ends with precisely $0\leq d'<m'$ consecutive $0$'s while $P''$ starts with precisely $0\leq d''<m''$ consecutive $0$'s.
	Then for every $d$ which satisfies
	\begin{equation}\label{eq:overlap}	
		\max\{d',d''\}\leq d\,,
	\end{equation}
	the mass of $P$ which lies above $d$ satisfies 
	$$
	\esc_P(d) = 	\frac{m'}{m'+m''}\esc_{P'}(d)+\frac{m''}{m'+m''}\esc_P(d)\,.
	$$
\end{lemma}
\begin{proof}
	For every $d$ as in \eqref{eq:overlap}, the height functions satisfy 
	\begin{equation}\label{eq:PP}
		\begin{aligned}
			\height_{P'}(i) &> d \; \textrm{if and only if} \;	\height_{P}(i) > d\,,\; \textrm{for every} \; 0\leq i<m'\,,\; \textrm{and}\\
			\height_{P''}(i) &> d \; \textrm{if and only if} \;	\height_{P}(m'+i) > d\,,\; \textrm{for every} \; 0\leq i<m''\,.
		\end{aligned}
	\end{equation}
	The first line in \eqref{eq:PP} will be proved and the second line follows by symmetry. Since $P'$ is a subsequence of $P$, it satisfies $\height_{P}(i) \leq \height_{P'}(i)$. Therefore, it is enough to show that if $\height_{P'}(i) > d$, then $\height_{P}(i) > d$. By assumption, $P(m'-d')\neq0$, so $\height_{P}(i)=\height_{P'}(i)$ for every $0\leq i\leq m'-d'$, and it is enough to consider $i>m'-d'$. Moreover, for any $m'-d' < i < m'$, the height function satisfies $$\height_{P'}(i) = i+1 - (m'-d')\leq d'\,.$$ 
	Since $d'\leq d$ by \eqref{eq:overlap}, this finishes the argument for \eqref{eq:PP}. Using \eqref{eq:PP} gives
	\ignore{\begin{align*}
		&\#\left\{0\leq i<m'+m'' \sep \max\left\{\height_P(i),\height_P(i+1)\right\}>d\right\}
		\\
		&=\#\left\{0\leq i<m' \sep \max\left\{\height_{P'}(i),\height_{P'}(i+1)\right\}>d\right\} + \#\left\{m'\leq i<m'' \sep \max\left\{\height_{P''}(i),\height_{P''}(i+1)\right\}>d\right\}
		\\
	\end{align*}}
	\begin{align*}
		\esc_P(d) &= \frac{\#\left\{0\leq i<m'+m'' \sep \max\left\{\height_P(i),\height_P(i+1)\right\}>d\right\}}{m'+m''}
		\\
		&=
		\frac{m'}{m'+m''}\frac{\#\left\{0\leq i<m' \sep \max\left\{\height_{P'}(i),\height_{P'}(i+1)\right\}>d\right\}}{m'} + 
		\\
		&\phantom{=} \frac{m''}{m'+m''}\frac{\#\left\{0\leq i<m'' \sep \max\left\{\height_{P''}(i),\height_{P''}(i+1)\right\}>d\right\}}{m''}
		\\
		&=
		\frac{m'}{m'+m''}\esc_{P'}(d) + \frac{m''}{m'+m''}\esc_{P''}(d)
		\,.
	\end{align*}
	\ignore{\item Since $|i'-i''|\leq m$,
		\begin{align*}
			\height_Q(i)&>d\,, &&\textrm{ if and only if } &\height_P(i)&>d\,, \textrm{ for every $0\leq i< m'$ and }\,,\\
			\height_Q(i+m+m')&>d\,, &&\textrm{ if and only if } &\height_{P'}(i)&>d\,, \textrm{ for every $0\leq i< m''$}\,.\\
		\end{align*}
		Therefore,
		\begin{align*}
			&\#\left\{0\leq i<m+m'+m'' \sep \max\left\{\height_Q(i),\height_Q(i+1)\right\}>d\right\}
			\\
			&=\#\left\{0\leq i<m' \sep \max\left\{\height_{P'}(i),\height_{P'}(i+1)\right\}>d\right\} + \#\left\{m'\leq i<m'' \sep \max\left\{\height_{P''}(i),\height_{P''}(i+1)\right\}>d\right\}
			\\
		\end{align*}}
\end{proof}

Finally, we will also need the computation of the escape of mass for a basic building block which is composed out of a stretch of $0$'s between nonzero elements. 
\begin{lemma}\label{lem:basic}
	For any $r>0$ let $P:[0,r]\to \fq$ be a sequence which satisfies $P(i) = 0$ if and only if $0<i<r$. 
	Then for every $d$, the escape of mass of $P$ satisfies
	\begin{equation}\label{eq:motivational-escape}
		\esc_{P}(d) = \frac{2\max\{r/2-d,0\}}{r}\,.  
	\end{equation}
\end{lemma}
\begin{proof}
	As in \eqref{eq:motivational}, the height function of $P_r$ satisfies 
	\begin{equation*}
		\height_{P}(i)= 
		\begin{cases}
			i & \mbox{ for } 0\le i\le r/2\\
			r - i &\mbox{ for } r/2 < i \le r.
		\end{cases}
	\end{equation*}
	Therefore, for every $d$,
	\begin{align*}
		\esc_{P}(d) &= \frac{\#\left\{0\leq i< r\sep \height_{P}(i)>d \right\}}{r} \\
		& = \frac{2\#\left\{0\leq i \leq r/2\sep \height_{P}(i)>d \right\}}{r} 
		\\
		&= \frac{2\max\{r/2-d,0\}}{r}\,.
	\end{align*}
	as claimed.
\end{proof}

\subsection{Partial escape of mass}

We will show now that for $\theta=\{b_n\}_{n=1}^\infty$ thought of also as the Laurent series $\theta=\sum_{n=1}^\infty b_nt^{-n}\in\fqttt$, Definition~\ref{def:escape-for-sequences} applied to the first column of $\frakf(\theta)$ agrees with Definition \ref{def:escape}. Recall that \eqref{eq:diophantine} defines the integers
\begin{equation*}
	i_{k} = \sum_{k'=0}^kd_{k'} 
\end{equation*}
for every $k\geq1$, and set $i_0=i_0(\theta)\defeq0$. For every $k\geq1$ and $d\geq0$, the mass of $\theta$ which lies above $d$ up to the $k$\textsuperscript{th} degree is defined in \eqref{eq:escape} as 
$$
\es_{\theta}(d,k) 
= 	\frac{\sum_{k'=1}^{k}\max\left\{d_{k'}-d,0\right\}}
{\sum_{k'=1}^{k}d_{k'}}\,.
$$

\begin{proposition}\label{prop:escape}
	For any $\theta=\{b_n\}_{n=1}^\infty$, 
	\begin{equation*}
		\es_{\theta}(d,k) = \esc_{\left.\frakf(\theta)\right|_{\left[0,2i_k\right]\times\{1\}}}(d)\,.
	\end{equation*}
\end{proposition}
\begin{proof}
	The standard facts from Diophantine approximation of Laurent series quoted in \eqref{eq:last} and \eqref{eq: frakf tile} mean that
	\begin{equation}\label{eq:last1}
		\frakf_{2i,1}(\theta)\neq0 \quad\iff\quad i=i_k \text{ for some }k\geq0\,,
	\end{equation}
	where the case $k=0$ holds since $\frakf_{0,1}=1$ by definition. Fix any $k\geq1$, and consider the first column of the tiling $\frakf$ between row $0$ to row $2i_k$. For the purpose of the current proof it is convenient to use the shortened notation for the height function along this column:
	\[
	\height \defeq \height_{\left.\frakf(\theta)\right|_{\left[0,2i_k\right]\times\{1\}}}\,.
	\]
	Then \eqref{eq:last1} implies that for every $d\geq0$,
	\begin{equation}\label{eq:aligned}
		\begin{aligned}
			&\#\left\{0\leq i<2i_k \sep \max\left\{\height(i),\height(i+1)\right\}>d\right\}
			\\
			&=\sum_{k'=1}^k\#\left\{2i_{k'-1}\leq i<2i_{k'} \sep \max\left\{\height(i),\height(i+1)\right\}>d\right\}
			\\
			&=\sum_{k'=1}^k2\max\{d_{k'}-d,0\}\,.
		\end{aligned}
	\end{equation}
	Therefore, the mass of $\theta$ which lies above $d$ up to the $k$\textsuperscript{th} degree as defined in \eqref{eq:escape} satisfies
	\begin{equation}
		\begin{aligned}
			\es_{\theta}(d,k) 
			&= 	\frac{\sum_{k'=1}^{k}\max\left\{d_{k'}\left(\theta\right)-d,0\right\}}
			{\sum_{k'=1}^{k}d_{k'}\left(\theta\right)}
			\\[2pt] 
			&= 
			\frac{2\sum_{k'=1}^{k}\max\left\{d_{k'}\left(\theta\right)-d,0\right\}}
			{2i_k}
			= \esc_{\left.\frakf(\theta)\right|_{\left[0,2i_k\right]\times\{1\}}}(d)\,.
		\end{aligned}
	\end{equation}
\end{proof}

Note that for every $j\geq0$, the number wall of $\theta=\{b_{n+j}\}_{n=1}^\infty$, which is the sequence $\theta$ shifted to the left $j$ times, is given by a shift of the indices of the number wall of $\theta$ by the vector $(j,j)$. Formally, every $m$ and $n$ satisfy
$$
\frakf_{m,n}\left(t^j\theta\right) = \frakf_{m+j,n+j}(\theta)\,.
$$
This follows from the definition of $\frakf$ in terms of Hankel or Toeplitz determinants. Therefore, Proposition~\ref{prop:escape} implies that for every $j\geq0$, the escape of mass of $t^j\theta$ satisfies 
\begin{equation}\label{eq:follows}
	\es_{t^j\theta}(d,k) = \esc_{\left.\frakf(\theta)\right|_{\left[j,2i_k(t^j\theta)+j\right]\times\{j+1\}}}(d)\,.
\end{equation}

The rest of this section deals with the number wall of the Thue-Morse sequence $\frakf=\frakf(\tm)$, which as we proved in Section~\ref{sec:structure} is (after ignoring its $0$\textsuperscript{th} row) equal to $\kappa_1\left(\sigma^\infty(\splus{0})\right)$. Our immediate goal is to use the structure we developed in that section in order to express the escape of mass along each of its columns in simple terms. In the following theorem we will show that these columns are periodic. Furtheremore, the period will be described explicitly and this will give a useful expression of the escape of mass in terms of the period. It is convenient to introduce the following shortened notation: for any $d,l\geq0$ and $0\leq j \leq 2^{l+2}$ denote
\[
\esc_{l,j}(d) \defeq \esc_{\left.\kappa_1\left(\sigma^l\left(\tiler{a}\right)\right)\right|_{\left[1,2^{l+2}\right]\times\{j+1\}}}(d)\,.
\]

\begin{theorem}\label{thm:escape}
	For every $d\geq5$, $l\geq0$ and $0 \leq j \leq 2^{l+2}$, the limit as $k$ goes to infinity of the mass of $t^j\theta$ that lies above $d$ up to the $k$\textsuperscript{th} degree exists, and satisfies
	\[
	\lim_{k\to\infty}\es_{t^j\tm}(d,k) = \frac{1}{3}\esc_{l,j}(d) + \frac{2}{3}\esc_{l+1,j}(d)\,.
	\]
\end{theorem}

\begin{proof}
	Fix $l\geq0$ and $0 \leq j \leq 2^{l+2}$. By Lemma \ref{lem:sigmal0}, the first column of $\sigma^{\infty}(\splus{0})$ is periodic with period
	\begin{align*}
		\begin{array}{c} \splus{0}\\ \tiler{a}\\ \sminus{1}\end{array}.
	\end{align*}
	Consequently, for 
	\[
	j'\defeq\min\left\{\left\lfloor\frac{j}{4}\right\rfloor,2^{l}-1\right\},
	\] 
	the $(j')$\textsuperscript{th} column of $\sigma^{\infty}(\splus{0})$ is periodic with period 
	\begin{align*}
		\begin{array}{c} \left.\sigma^l(\splus{0})\right|_
			{\left[1,2^{l}\right]\times\{j'+1\}}\\
		\left.\sigma^l\left(\tiler{a}\right)\right|_
		{\left[1,2^{l}\right]\times\{j'+1\}}\\
		\left.\sigma^l(\sminus{1})\right|_
		{\left[1,2^{l}\right]\times\{j'+1\}}\\
		\end{array},
	\end{align*}
	and the $j$\textsuperscript{th} column of $\frakf(\tm)$ is periodic with period (due to Lemma \ref{lem:consistent} this holds also for $j=2^{l+2}$)
	\begin{equation}\label{eq:jth-column}
		\begin{aligned}
		\begin{array}{c}
			\left.\kappa_1\left(\sigma^l(\splus{0})\right)\right|_
			{\left[1,2^{l+2}\right]\times\{j+1\}}\\
			\left.\kappa_1\left(\sigma^l\left(\tiler{a}\right)\right)\right|_
			{\left[1,2^{l+2}\right]\times\{j+1\}}\\
			\left.\kappa_1\left(\sigma^l(\sminus{1})\right)\right|_
			{\left[1,2^{l+2}\right]\times\{j+1\}}\\
		\end{array}.
		\end{aligned}
	\end{equation}
	Recall that by definition $\frakf_{j,j+1}=1$ for every $j\geq0$. Since the period in \eqref{eq:jth-column} is of length $3\cdot2^{l+2}$, this implies that $\frakf_{j+3\cdot2^{l+2}i,j+1}=1$ for every $i\geq0$.
	Therefore,
	\ignore{By Theorem \ref{thm:main2},
	\[
	\left\{\frakf_{i,j}(\tm)\right\}_{i=1,j=1}^\infty=\kappa_1\left(\sigma^\infty(\splus{0})\right).
	\]}
	it follows from \eqref{eq:follows} and Lemma~\ref{lem:gluing} with $d'=d''=0$, that for every $d\geq0$,
	\begin{align*}
		\lim_{k\to\infty}\esc_{t^j\tm}(d,k) &= 	\lim_{k\to\infty}\esc_{\left.\frakf(\tm)\right|_{\left[j,3\cdot2^{l+2}k+j\right]\times\{j+1\}}}(d) \\
		&=
		\esc_{\left.\frakf(\tm)\right|_{\left[j,3\cdot2^{l+2}+j\right]\times\{j+1\}}}(d)\,.
	\end{align*}
	The border rows between $\sigma^l(\splus{0})$ and $\sigma^l\left(\tiler{a}\right)$ were analysed in the proof of Theorem \ref{thm:main2} in Section \ref{sec:tiling}. It follows from the discussion there that every sequence of consecutive $0$'s of length strictly greater than $5$ in this column is contained either within
	$$
	\left.\kappa_1\left(\sigma^l\left(\tiler{a}\right)\right)\right|_{\left[1,2^{l+2}\right]\times\{j+1\}},\\
	$$
	or within
	\begin{equation*}
		\begin{aligned}
			\begin{array}{c}
				\left.\kappa_1\left(\sigma^l(\sminus{1})\right)\right|_{\left[1,2^{l+2}\right]\times\{j+1\}}\\
				\left.\kappa_1\left(\sigma^l(\splus{0})\right)\right|_{\left[1,2^{l+2}\right]\times\{j+1\}}
			\end{array}.
		\end{aligned}
	\end{equation*}
	Therefore, applying Lemma~\ref{lem:gluing} twice gives that
	\begin{equation*}
	\esc_{\left.\frakf(\tm)\right|_{\left[j,3\cdot2^{l+2}+j\right]\times\{j+1\}}}(d) = 
	\esc_{\left.\frakf(\tm)\right|_{\left[2^{l+2},4\cdot2^{l+2}\right]\times\{j+1\}}}(d) \\\,,
	\end{equation*}
	for every $d\geq5$. Finally, recalling \eqref{eq:sigmal8} and applying Lemma~\ref{lem:gluing} again proves that 
	$$
	\esc_{\left.\frakf(\tm)\right|_{\left[2^{l+2}+1,4\cdot2^{l+2}\right]\times\{j+1\}}}(d) = \frac{1}{3}\esc_{l,j}(d) + \frac{2}{3}\esc_{l+1,j}(d)\,.
	$$
\end{proof}

\begin{proof}[Proof of Theorem \ref{thm:main}]
	By Lemma~\ref{lem: kappa1 of r0} and Lemma~\ref{lem:basic}, every $d,l\geq0$ satisfy
	\[
	\esc_{l,2^{l+1}}(d) = \frac{2\max\left\{2^{l+1}-2-d,0\right\}}{2^{l+2}}\,,
	\]
	and if $d\geq2$ then 
	\[
	\esc_{l,2^{l+2}}(d) = 0\,.
	\]
	Therefore, applying Theorem~\ref{thm:escape} with $j=2^{l+2}$ and taking the limit $l\to\infty$ gives \eqref{eq:main}.
\end{proof}

\subsection{Escape of mass along more general shifts}
The theory developed thus far allows for a slightly more detailed analysis of the escape of mass along subsequences of shifts of $\tm$.
In fact, applying 
Corollary \ref{cor:key} to \eqref{eq:sigmal8} gives that for every $l\geq0$,
\begin{align}
	\sigma^{l+2}\left(\tiler{a}\right)&=
	\mat{
        \sigma^{l+1}(\sminus{1})&\sigma^{l+1}(\splus{2})\\
        \sigma^{l+1}(\splus{0})&\sigma^{l+1}(\sminus{3})
    } 
    \nonumber
    \\[2pt]
    \label{eq:twice}
    &=
    \mat{
        \sigma^{l}\left(\tiler{a}\right)&\sigma^{l}(\splus{1})&\sigma^{l}(\sminus{2})&\sigma^{l}\left(\tiler{a}\right)
        \\
        \sigma^{l}(\sminus{1})&\sigma^{l}(\tilec{1})&\sigma^{l}(\tilec{2})&\sigma^{l}(\splus{2})
        \\
        \sigma^{l}(\splus{0})&\sigma^{l}(\tilec{0})&\sigma^{l}(\tilec{3})&\sigma^{l}(\sminus{3})
        \\
        \sigma^{l}\left(\tiler{a}\right)&\sigma^{l}(\sminus{0})&\sigma^{l}(\splus{3})&\sigma^{l}\left(\tiler{a}\right)
    }.
\end{align}
\noindent
Note that the submatrix given by the coordinates $(1,2),(1,3);(4,2),(4,3)$ of \eqref{eq:twice} is precisely 
\begin{equation*}
	\sigma^{l+1}\left(\tiler{b}\right)=
	\mat{
		\sigma^{l}(\splus{1})&\sigma^{l}(\sminus{2})\\
		\sigma^{l}(\sminus{0})&\sigma^{l}(\splus{3})
	}.
\end{equation*}
This enables a simple recursive formula for $\esc_{l+2,j}$ in terms of $\esc_{l+1,j}$ and $\esc_{l,j}$. But first, a lemma that relates between the escape of mass of $\sigma^{l}\left(\tiler{a}\right)$ and $\sigma^{l}\left(\tiler{b}\right)$ is required. Define 
\[
\overline{e}_{l,j}(d) \defeq \esc_{\left.\kappa_1\left(\sigma^l\left(\tiler{b}\right)\right)\right|_{\left[1,2^{l+2}\right]\times\{j+1\}}}(d)\,.
\]

\begin{lemma}
	For every $d\geq3$, $l\geq0$ and $0 \leq j \leq 2^{l+2}$,
	$$\overline{\esc}_{l,j}(d) = \esc_{l,j}(d)\,.$$
\end{lemma}

\begin{proof}
	The result follows from the fact that both $\kappa_1\left(\sigma^l\left(\tiler{a}\right)\right)$ and $\kappa_1\left(\sigma^l\left(\tiler{b}\right)\right)$ contain the same diagonally-aligned square patterns of $0$'s of side length $d$ for every $d\geq3$. 
	Recall Lemma~\ref{lem: kappa1 of r0} and note that it has an analogue with 
	$\tiler{b}$ replacing $\tiler{a}$.
	In fact, for every $2\leq l'\leq l$, every diagonally-aligned square of $0$'s of size $2^{l'}-2 < d \leq 2^{l'+1}-2$ which is contained in either $\kappa_1\left(\sigma^l\left(\tiler{a}\right)\right)$ or $\kappa_1\left(\sigma^l\left(\tiler{b}\right)\right)$ is contained in the $0$'s which lie in the centre of $\kappa_1\left(\sigma^{l'}\left(\tiler{a}\right)\right)$and $\kappa_1\left(\sigma^{l'}\left(\tiler{b}\right)\right)$, respectively, for some $\tiler{a}$ or $\tiler{b}$ contained in $\sigma^{l-l'}\left(\tiler{a}\right)$ and $\sigma^{l-l'}\left(\tiler{b}\right)$. This can be proved by induction. Since $\tiler{b}=\iota\tiler{a}$ and $$\sigma^{l-l'}\left(\tiler{b}\right)=\iota\sigma^{l-l'}\left(\tiler{a}\right) = \tinn_\iota\left(\sigma^{l-l'}\left(\tiler{a}\right)\right),$$ 
	it follows that $\sigma^{l-l'}\left(\tiler{a}\right)(m,n)\in\{\tiler{a},\tiler{b}\}$ if and only if $\sigma^{l-l'}\left(\tiler{b}\right)(m,n)\in\{\tiler{a},\tiler{b}\}$, which finishes the proof.
\end{proof}
	
\begin{theorem}\label{thm:mkj}
	For every $d\geq0$, $l\geq0$, and $0 \leq j \leq 2^{l+2}$,
	\begin{equation}\label{eqn:m_k,j2}
		\esc_{l,j}(d)= \esc_{l,2^{l+2} - j}(d)\,,
	\end{equation}
	and if $d\geq5$ and $l\geq2$, then
	\begin{equation}\label{eqn:m_k,j}
		\esc_{l,j}(d)=
		\begin{cases}
			\frac{1}{2}\esc_{l-1,j}(d)+\frac{1}{2}\esc_{l-2,j}(d)&0\leq j\leq 2^{l}
			\\[10pt]
			\frac{2^l-\max\{d+2-\left(j-2^{l}\right),0\}}{2^{l+1}}+\frac{1}{2}\esc_{l-1,j-2^{l}}(d)&2^{l}\leq j\leq 2^{l+1}
		\end{cases}.
	\end{equation}
\end{theorem}
\begin{proof}
	Equation \eqref{eqn:m_k,j2} follows immediately from the fact that $\rho^2\tiler{a}=\tiler{a}$ combined with 
	Lemma~\ref{lem: kappa equivariance}. Equation \eqref{eqn:m_k,j} follows by inspecting \eqref{eq:twice} and applying lemmas \ref{lem:gluing} and \ref{lem:basic}.
\end{proof}

In order to complement 
Theorem \ref{thm:main}, 
an example of a subsequence which exhibits full escape of mass is provided:
\begin{corollary}\label{cor}
    Let 
    \begin{equation}\label{eq:jl}
    	j_l=\sum_{l'=0}^{\left\lceil\frac{l}{2}\right\rceil}2^{\max\{l-2l',0\}}\,.
    \end{equation} 
    Then, every $d\geq0$ satisfies
    \[
    \lim_{l\to\infty}\lim_{k\to\infty}\esc_{t^{j_l}\tm}(d,k)=1
    \]
\end{corollary}
\begin{proof}
    Equation \eqref{eq:jl} implies that every $l\geq2$ satisfies
    \begin{equation}\label{eq:jl-symmetry}
    	j_l = 2^{l+1} - j_{l-1} = 2^l + j_{l-2}\,.
    \end{equation}
    Fix any $d\geq5$. For every $l\geq3$, Theorem \ref{thm:mkj} gives 
     \begin{align*}
    	\esc_{l,j_l}(d) 
    	={}&
    	\frac12\left(1-\frac{\max\{d+2-\left(j-2^{l}\right),0\}}{2^l}\right) + 
    	\\
    	&
    	\frac{1}{2}\esc_{l-1,j_{l-2}}(d) 
    	\\[10pt]
    	={}&
  		\frac12\left(1-\frac{\max\{d+2-\left(j-2^{l}\right),0\}}{2^l}\right) + 
  		\\
  		&
  		\frac{1}{4}\esc_{l-2,j_{l-2}}(d) +
  		\frac{1}{4}\esc_{l-3,j_{l-2}}(d)
  		\\[10pt]
  		={}&
  		\frac12\left(1-\frac{\max\{d+2-\left(j-2^{l}\right),0\}}{2^l}\right) +
  		\\
  		&
  		\frac18\left(1-\frac{\max\{d+2-\left(j-2^{l-2}\right),0\}}{2^{l-2}}\right)+
  		\\
  		&
  		\;\frac18\esc_{l-2,j_{l-2}}(d) +
  		\;\frac14\esc_{l-3,j_{l-3}}(d)\,.
    \end{align*}
    Hence,
    $$\lim_{l\to\infty}\esc_{l,j_l}(d)=1\,.$$
    Since $2^l\leq j_l\leq 2^{l+1}$, Theorem \ref{thm:escape} together with \eqref{eqn:m_k,j2} and \eqref{eq:jl-symmetry} imply that
    \begin{align*}
    \lim_{k\to\infty}\esc_{t^{j_l}\tm}(d,k) &= \frac{1}{3}\esc_{l-1,j_l}(d) + \frac{2}{3}\esc_{l,j_l}(d)
    \\
    &=
    \frac{1}{3}\esc_{l-1,j_{l-1}}(d) + \frac{2}{3}\esc_{l,j_l}(d) 
    \xrightarrow[{l\to\infty}]{}1
    \,.
    \end{align*}
\end{proof}


\paragraph{}
\renewcommand{\abstractname}{Acknowledgements}
\begin{abstract}
	EN wishes to thank Elon Lindenstrauss for his encouragement to pursue research of the utility of the frame relations in homogeneous dynamics over function fields, and for an inspiring discussion about rotating number walls by $45\degree$. EN wishes to thank his father, Gozal Nesharim, for his interest in this project and for a useful discussion about the self-similarity of \eqref{eq:tm_flat_array} and $3$-dimensional solid realisations of it. The authors thank Omri Solan for his observation that the limit in \eqref{eq:main} is $1$ if $\liminf$ is replaced with $\limsup$, during the workshop ``Distribution of orbits:
	Arithmetics and Dynamics'' where this work was presented (see Corollary~\ref{cor} for an explicit subsequence which attains this $\limsup$).
\end{abstract}

\bibliography{Ref} 

\begin{bibdiv}
\begin{biblist}

\bib{ALN}{article}{
      author={Adiceam, Faustin},
      author={Nesharim, Erez},
      author={Lunnon, Fred},
       title={On the t-adic {L}ittlewood conjecture},
        date={2021},
     journal={Duke Mathematical Journal},
      volume={170(10)},
       pages={2371\ndash 2419},
}

\bib{APWW}{article}{
      author={Allouche, Jean-Paul},
      author={Peyriere, Jacques},
      author={Wen, Zhi-Xiong},
      author={Wen, Zhi-Ying},
       title={Hankel determinants of the {T}hue-{M}orse sequence},
        date={1998},
     journal={Annales de L’Institute Fourier},
       pages={1\ndash 27},
}

\bib{AllSha}{book}{
      author={Allouche, Jean-Paul},
      author={Shallit, Jeffrey},
       title={Automatic sequences: Theory, applications, generalizations},
   publisher={Cambridge: Cambridge University Press},
        date={2003},
}

\bib{AkaShapira}{article}{
      author={Aka, Menny},
      author={Shapira, Uri},
       title={On the evolution of continued fractions in a fixed quadratic
  field},
        date={2018},
        ISSN={0021-7670,1565-8538},
     journal={J. Anal. Math.},
      volume={134},
      number={1},
       pages={335\ndash 397},
         url={https://doi.org/10.1007/s11854-018-0012-4},
      review={\MR{3771486}},
}

\bib{dMT}{article}{
      author={de~Mathan, Bernard},
      author={Teulié, Oliver},
       title={Problèmes diophantiens simultanés},
        date={2004},
     journal={Monatsh. Math},
      volume={143},
       pages={229–245},
}

\bib{GR}{misc}{
      author={Garrett, Samuel},
      author={Robertson, Steven},
       title={Counterexamples to the $p(t)$-adic {L}ittlewood conjecture over
  small finite fields},
        date={2024},
         url={https://arxiv.org/abs/2405.14454},
}

\bib{KPS}{article}{
      author={Kemarsky, Alex},
      author={Paulin, Frederic},
      author={Shapira, Uri},
       title={Escape of mass in homogeneous dynamics in positive
  characteristic},
        date={2017},
     journal={Journal of Modern Dynamics},
      volume={11},
       pages={369\ndash 407},
}

\bib{L}{article}{
      author={Lunnon, Fred},
       title={The number-wall algorithm: an {LFSR} cookbook},
        date={2001},
     journal={Journal of Integer Sequences},
      volume={4},
}

\bib{L09}{inproceedings}{
      author={Lunnon, Fred},
       title={The {P}agoda sequence: a ramble through linear complexity, number
  walls, {D}0{L} sequences, finite state automata, and aperiodic tilings},
        date={2009},
   booktitle={Proceedings {I}nternational {W}orkshop on {T}he {C}omplexity of
  {S}imple {P}rograms},
      series={Electron. Proc. Theor. Comput. Sci. (EPTCS)},
      volume={1},
   publisher={EPTCS, [place of publication not identified]},
       pages={130\ndash 148},
         url={https://doi.org/10.4204/EPTCS.1.13},
}

\bib{PS}{article}{
      author={Paulin, Frederic},
      author={Shapira, Uri},
       title={On continued fraction expansions of quadratic irrationals in
  positive characteristic},
        date={2018},
     journal={Groups Geometry and Dynamics},
      volume={14},
       pages={81–105},
}

\bib{dT04}{book}{
      author={Thakur, Dinesh~S.},
       title={Function field arithmetic},
   publisher={World Scientific Publishing Co., Inc., River Edge, NJ},
        date={2004},
        ISBN={981-238-839-7},
         url={https://doi.org/10.1142/9789812562388},
}

\end{biblist}
\end{bibdiv}
\bibliographystyle{alpha}
\end{document}